\documentclass[12pt,draftcls,onecolumn]{IEEEtran}

\ifCLASSINFOpdf

\else

\fi

\usepackage{ifpdf}
\ifpdf
    \usepackage{graphicx}
    \usepackage{epstopdf}
\else
    \usepackage{graphicx}

\fi

\usepackage{mathptmx} 
\usepackage{times} 
\usepackage{amsmath} 
\usepackage{amssymb}  
\usepackage{tabularx}
\usepackage{algorithm}
\usepackage{algorithmic}
\usepackage{enumerate}
\usepackage{empheq} 
\newcommand*\widefbox[1]{\fbox{\hspace{1em}#1\hspace{1em}}}
\usepackage[dvipsnames,usenames]{color}
\usepackage[colorlinks,
linkcolor=BrickRed,
filecolor=BrickRed,
citecolor=RoyalPurple,
]{hyperref}
\usepackage[font=small]{caption}
\usepackage{subfigure}

\usepackage{color}

\DeclareSymbolFont{newfont}{OML}{cmm}{m}{it}
\DeclareMathSymbol{\Epsilon}{3}{newfont}{15}
\DeclareMathSymbol{\Varrho}{3}{newfont}{37}
\usepackage{tabularx}

\title{\LARGE \bf
A Controlled Particle Filter for Global Optimization}

\author{Chi Zhang, Amirhossein Taghvaei and Prashant G. Mehta
\thanks{C. Zhang, A. Taghvaei and P. G. Mehta are with the Coordinated Science Laboratory and
the Department of Mechanical Science and Engineering at the University of
Illinois at Urbana-Champaign (UIUC)
{\tt\small czhang54@illinois.edu; taghvae2@illinois.edu; mehtapg@illinois.edu}}
\thanks{Financial support from the NSF CMMI grants 1334987 and 1462773 is gratefully acknowledged.}%
\thanks{The conference versions of this paper appear in~\cite{Chi_CDC13}.}
}

\def\qed{\hspace*{\fill}~\IEEEQED\par\endtrivlist\unskip}

\def\Re{\mathbb{R}}

\def\argmin{\mathop{\text{\rm arg\,min}}}
\def\ind{\text{\rm\large 1}}
\def\varble{\,\cdot\,}

\usepackage{eufrak}

\def\varble{\,\cdot\,}

\def\Sec#1{Sec.~\ref{#1}}

\def\notes#1{\marginpar{\tiny #1}\typeout{Notes!
Notes!
Notes!
}}
\renewcommand{\notes}[1]{\typeout{notes!}}

\def\bbbone{{\mathchoice {\rm 1\mskip-4mu l} {\rm 1\mskip-4mu l}
{\rm 1\mskip-4.5mu l} {\rm 1\mskip-5mu l}}}
\def\ind{\bbbone}

\newcommand{\field}[1]{\mathbb{#1}}
\newcommand{\tr}{\mbox{tr}}

\def\Re{\field{R}}

\def\Sec#1{Sec.~\ref{#1}}

\def\clN{{\cal N}}

\def\calP{{\cal P}}

\def\K{{\sf K}} 
 
\def\Dt{\Delta t}

\def\R{\mathbb{R}} 

\newcommand{\kepsN}{{k}_{\epsilon}^{(N)}}

\newcommand{\phiepsN}{{\phi}^{(N)}_\epsilon}

\newcommand{\phiM}{\phi^{(M)}}
\newcommand{\phiMN}{\phi^{(M,N)}}

\newcommand{\TepsN}{{T}^{(N)}_\epsilon}

\newcommand{\nepsN}{{n}^{(N)}_\epsilon}

\newcommand{\geps}{{g}_{\epsilon}}

\def\exp{\text{exp}}

\makeatletter

\newcommand{\Rom}[1]{\expandafter\@slowromancap\romannumeral #1@}
\makeatother

\newcounter{rmnum}

\newenvironment{romannum}{\begin{list}{{\upshape (\roman{rmnum})}}{\usecounter{rmnum}
\setlength{\leftmargin}{6pt}
\setlength{\rightmargin}{4pt}
\setlength{\itemindent}{-1pt}
}}{\end{list}}

\newcounter{anum}

\newcommand{\ud}{\,\mathrm{d}}

\newcommand{\pr}{\rho}

\newcommand{\expect}[1]{{\sf E}[#1]}
\newcommand{\newP}[1]{\medskip\noindent{\bf #1:}}

\usepackage{ulem}

\newtheorem{theorem}{Theorem}
\newtheorem{corollary}{Corollary}
\newtheorem{proposition}{Proposition}
\newtheorem{lemma}{Lemma}

\newtheorem{example}{Example}
\newtheorem{remark}{Remark}

\newtheorem{proof}{Proof}

\def\Sec#1{Sec.~\ref{#1}}

\renewcommand{\baselinestretch}{1.28}

\makeatletter
\newsavebox\myboxA
\newsavebox\myboxB
\newlength\mylenA

\newcommand*\xoverline[2][0.75]{%
    \sbox{\myboxA}{$\m@th#2$}%
    \setbox\myboxB\null
    \ht\myboxB=\ht\myboxA%
    \dp\myboxB=\dp\myboxA%
    \wd\myboxB=#1\wd\myboxA
    \sbox\myboxB{$\m@th\overline{\copy\myboxB}$}
    \setlength\mylenA{\the\wd\myboxA}
    \addtolength\mylenA{-\the\wd\myboxB}%
    \ifdim\wd\myboxB<\wd\myboxA%
       \rlap{\hskip 0.5\mylenA\usebox\myboxB}{\usebox\myboxA}%
    \else
        \hskip -0.5\mylenA\rlap{\usebox\myboxA}{\hskip 0.5\mylenA\usebox\myboxB}%
    \fi}
\makeatother

\def\bN{\bar{N}}

\begin{document}
\normalem
\maketitle

\vspace{-0.5in}
\begin{abstract}

A particle filter is introduced to numerically approximate a solution of
the global optimization problem.  The theoretical significance of this
work comes from its variational aspects: (i) the proposed particle
filter is a controlled interacting particle system where the control
input represents the solution of a mean-field type optimal control
problem; and (ii) the associated density transport is shown to be a
gradient flow (steepest descent) for the optimal value function, with
respect to the Kullback--Leibler divergence.  The optimal control
construction of the particle filter is a significant departure from
the classical importance sampling-resampling based approaches.    
There are several practical advantages: (i) resampling, reproduction,
death or birth of particles is avoided; (ii) simulation variance can
potentially be reduced by applying feedback control principles; and
(iii) the parametric approximation naturally arises as a special case.
The latter also suggests systematic approaches for numerical
approximation of the optimal control law.  The theoretical results are
illustrated with numerical examples.

\end{abstract}

\section{Introduction}
\label{sec:intro}

We consider the global optimization problem:
\begin{equation*}
 \min_{x \in \Re^d} ~h(x),
\end{equation*}
where $h:\Re^d\rightarrow\Re$ is a real-valued function.  This
paper is concerned with gradient-free simulation-based 
algorithms to obtain the {\em global minimizer} 
\[
\bar{x} = \argmin_{x\in \Re^d} \;\; h(x).
\] 
It is assumed that such a minimizer exists and is unique.

A Bayesian approach to solve the problem is as follows: Given an
everywhere positive initial density (prior) $p_0^*$, define the
(posterior) density at a positive time $t$ by
\begin{equation}
p^*(x,t) :=\frac{ p_0^*(x) \, \exp{(-\beta h(x)\,t )} }{\int p_0^*(y) \, \exp{(-\beta h(y) \,t)} \ud y},
\label{eqn:evolution}
\end{equation}
where $\beta$ is a positive constant parameter. 
Under certain additional technical assumptions on $h$ and 
$p_0^*$, the density $p^*(x,t)$ weakly converges to the Dirac delta measure at
$\bar{x}$ as time $t\rightarrow\infty$ (See Appendix~\ref{apdx:exact}).  
The Bayesian approach is attractive because it
can be implemented recursively:
Consider a finite time interval $[0,T]$ with an associated
discrete-time sequence $\{t_0,t_1, t_2,\hdots,t_{\bN}\}$ of sampling
instants, with $0=t_0 < t_1 < \ldots < t_{\bN} = T$, and increments given by $\Delta t_n \doteq t_n - t_{n-1}, n =
1,\ldots,\bN$. The posterior
distribution is expressed recursively as:
\begin{equation}
\begin{aligned}
\text{Initialization:}\quad\rho_0(x) & = p_0^*(x),\\
\text{Update:}\quad\rho_n(x) & = \frac{\rho_{n-1} (x) \, \exp(- \beta h(x) \Delta t_n) }{\int \rho_{n-1}
  (y) \, \exp(- \beta h(y) \Delta t_n) \ud y}.
\end{aligned}
\label{eq:Bayes_with_Yn}
\end{equation}
Note that $\{\rho_n\}$ is a sequence of probability densities.  At
time $t_n$, $\rho_n(x) = p^*(x,t_n)$ by construction.

A particle filter is a simulation-based algorithm to sample
from $\rho_n$.  A particle filter is comprised of $N$ stochastic processes $\{X^i_t : 1\le i \le N\}$. The vector $X^i_t \in
\Re^d$ is the state for the $i^{\text{th}}$ particle at time $t$. For
each time $t$, the empirical distribution formed by the ensemble is used to approximate the posterior  distribution.
The empirical distribution is defined for any measurable set $A\subset\Re^d$ by
\begin{equation*}
p^{(N)}(A,t) = \frac{1}{N}\sum_{i=1}^N \ind\{ X^i_t\in A\}. 
\label{e:piN}
\end{equation*}
A sequential importance sampling resampling (SISR)
particle filter implementation involves the following recursive steps:
\begin{equation}
\begin{aligned}
\text{Initialization:}\quad & X^i_0\;\; \mathop{\sim}^\text{{i.i.d.}}\;\; p_0^*,\\
\text{Update:}\quad & X^i_n \;\;\mathop{\sim}^{\text{i.i.d.}} \;\; \sum_{i=1}^N w^i_n
\delta_{X^{i}_{n-1}}, 
\end{aligned}
\label{eq:SIR}
\end{equation}
where $w^i_n \propto \exp(- \beta h(X^{i}_{n-1}) \Delta t_n)$ are referred
to as the importance weights and $\delta_{z}$ denotes the Dirac-delta at
$z\in\Re^d$.  In practice, the importance weights $w^i_n$ can
potentially suffer from large variance.  To address this problem, several
extensions have been described in literature based on consideration of suitable
sampling (proposal) distributions and efficient resampling schemes;
cf.,~\cite{doucet09,smith2013}.  

The use of probabilistic models to
derive recursive sampling algorithms is by now a standard
solution approach to the global optimization problem: The
model~\eqref{eqn:evolution} appears in~\cite{Wang10} with
closely related variants given in~\cite{CE99,hu2007model}.
Importance sampling type schemes, of the form~\eqref{eq:SIR}, based on these and more general (stochastic)
models appear
in~\cite{zhou2014particle,zhou2013sequential,liu2015posterior,Bin16,stinis2012stochastic}.

\medskip

In this paper, we present an alternate control-based approach to the construction and
simulation of the particle filter for global optimization.  
In our approach, the particle filter is a {\em controlled} interacting
particle system where the dynamics of the $i^{\text{th}}$ particle
evolve according to
\begin{equation}
\frac{\ud X_t^i }{\ud t} = u(X_t^i,t),\quad X_0^i\sim p_0^*,
\label{eqn:particle_dyn}
\end{equation}
where the {\em control function} $u(x,t)$ is
obtained by solving a weighted Poisson equation:
\begin{equation}
\label{eqn:EL_phi_intro}
\begin{aligned}
-\nabla \cdot (\rho(x) \nabla \phi(x)) & = (h (x)-\hat{h}) \rho(x),\quad x\in \Re^d,
\\
\int \phi (x) \rho(x) \ud x & = 0,
\end{aligned}
\end{equation}
where  $\hat{h} :=
\int h(x)\rho(x)\ud x$, $\nabla$ and $\nabla \cdot $ denote the
gradient and the divergence operators, respectively, and at time $t$, 
$\rho(x)=p(x,t)$ denotes the density of $X_t^i$\footnote{Although this
  paper is limited to $\Re^d$, the proposed algorithm is applicable to
  global optimization problems on differential manifolds, e.g., matrix Lie
  groups (For an intrinsic form of the Poisson equation, see~\cite{Chi_ACC16}).  
  For domains with boundary, the pde is accompanied by a Neumann boundary condition:
\[
\nabla \phi(x) \cdot {n}(x) = 0
\] 
for all $x$ on the boundary of the domain where ${n}(x)$ is a unit
normal vector at the boundary point $x$.}. 
In terms of the solution $\phi(x)$ of~\eqref{eqn:EL_phi_intro}, 
the control function at time $t$ is given by
\begin{equation}
u(x,t) = -\beta \nabla \phi(x). 
\label{eqn:gradient_gain}
\end{equation}
Note that the control function $u$ is vector-valued (with dimension
$d\times 1$) and it needs to be obtained for each value of time $t$.
The basic results on existence and uniqueness of $\nabla\phi$ are
summarized in Appendix~\ref{apdx:Poisson}.  These results require
additional assumptions on the prior $p_0^*$ and the function $h$.
These assumptions appear at the end of this section.

The inspiration for controlling a single particle -- via the control input $U_t^i$
in~\eqref{eqn:particle_dyn} -- comes from the mean-field type control
formalisms~\cite{huacaimal07,bens-master,brockett2007,YinMehMeyShan12}, control
methods for optimal transportation~\cite{chen2016Relation,chen2016Linear}, and the feedback
particle filter (FPF) algorithm for nonlinear
filtering~\cite{Tao_TAC,Tao_Automatica}.  One interpretation of the control input $u(X_t^i,t)$ is that it implements the ``Bayesian update
step'' to steer the ensemble $\{X_t^i:1\le i\le N\}$ towards
the global minimum $\bar{x}$.  Structurally, the control-based approach of this
paper is a significant departure from the importance sampling based
implementation of the Bayes rule in conventional particle filters.  It
is noted that there are no additional steps, e.g., associated
with resampling, reproduction, death, or birth of particles.  In the
language of importance sampling, the particle flow is designed so that
the particles automatically have identical importance weights for all
time.  The Poisson equation~\eqref{eqn:EL_phi_intro} also appears in
FPF~\cite{Tao_TAC} and in other related algorithms for nonlinear
filtering~\cite{reich11,Daum10}.

\medskip

The contributions of this paper are as follows:
\begin{romannum}
\item{\bf Variational formulation:} A time-stepping
procedure is introduced consisting of successive minimization
problems in the space of probability densities.  The construction
shows the density transport~\eqref{eqn:evolution} may be regarded as a gradient flow, or a
steepest descent, for the expected value of the function $h$, with
respect to the Kullback--Leibler divergence.  More significantly, the
construction is used to motivate a mean-field type optimal control
problem.  The control
law~\eqref{eqn:particle_dyn}-\eqref{eqn:gradient_gain} for the
proposed particle filter represents the solution to this problem.  The
Poisson equation~\eqref{eqn:EL_phi_intro} is derived from the
first-order analysis of the Bellman's optimality principle.  To the
best of our knowledge, our paper provides the first
derivation/interpretation of a (Bayesian) particle filter as a
solution to an optimal control problem.  For a discussion on the
importance of the variational aspects of nonlinear filter,
see~\cite{mitter03} and~\cite{laugesen15}.  

\medskip

\item{\bf Quadratic Gaussian case:} For a quadratic objective function $h$
  and a Gaussian prior $p_0^*$, the partial differential equation (pde)~\eqref{eqn:EL_phi_intro} admits a closed-form solution.
  The resulting control law is shown to be affine in the state.  The
  quadratic Gaussian problem is an example of the more general
  parametric case where the density is of a (known) parametrized form.  For
  the parametric case, the filter is shown to be equivalent to the
  finite-dimensional natural gradient algorithm for the parameters~\cite{kakade}.  

\medskip

\item{\bf Numerical algorithms:} For the general non-parametric case,
the pde~\eqref{eqn:EL_phi_intro} may not admit a closed form solution.  Based on our
prior work on the feedback particle filter,
two algorithms are discussed: (i) Galerkin
algorithm; and (ii) Kernel-based algorithm.  The
algorithms are completely adapted to data (That is, they do not require
an explicit approximation of $p(x,t)$ or computation of
derivatives of $h$).  Two numerical examples are presented to illustrate these
algorithms. 

\end{romannum}

\newP{Literature review} There are two broad categories of global optimization algorithms: (i) Instance-based algorithms and (ii)
Model-based algorithms; cf.,~\cite{Zlochin04}.  The instance-based algorithms
include simulated annealing~\cite{kirkpatrick1983SA, romeijn1994simulated}, genetic algorithms~\cite{GA}, 
nested partitioning methods~\cite{shi2000nested}, and various
types of random search~\cite{zabinsky2013stochastic} and particle
swarm~\cite{PSO,yang2010nature} algorithms.  The optimization is cast
as an iterative search where one seeks to balance the exploration of
the state-space with the optimization objective.  In~\cite{CE99}, such
algorithms are referred to as `local search heuristics,' presumably
because they depend upon the local topological structure of the
state-space. 

In recent years, the focus has been on model-based algorithms where a
probabilistic model -- sequence of recursively-defined distributions
(e.g.,~\eqref{eq:Bayes_with_Yn}) 
-- informs the search of the global optimizer.
Examples include (i) non-parametric approaches such as estimation of
distribution algorithm~\cite{larranaga2002estimation}, 
sequential Monte Carlo simulated annealing~\cite{zhou2013sequential}, 
and the particle filter optimization (PFO)~\cite{zhou2014particle}; and (ii) parametric approaches such as the
cross-entropy (CE)~\cite{CE99} and the model reference adaptive search
~\cite{hu2007model} algorithms.  Recent surveys of the model-based
algorithms appear in~\cite{Hu_survey12, Bin16}.  

The main steps for a non-parametric algorithm are as follows: (i) the
(prescribed) distribution at discrete time $t_n$ is used to generate a
large number of samples, (ii) a selection mechanism is used to generate
`elite samples' from the original samples, (iii) the distribution at
time $t_n+\Delta t_n$ is the distribution estimated from the elite
samples. The SISR particle filter~\eqref{eq:SIR} may be viewed as a
model-based algorithm where the selection mechanism is guided by the
importance weights and the new samples are generated via the
resampling step.  A more general version of~\eqref{eq:SIR} is the
model-based evolutionary optimization (MEO) algorithm~\cite{Wang10} where
the connection to the replicator pde is also provided.  The stochastic
extension of this algorithm is the PFO~\cite{zhou2014particle} based
on a nonlinear filtering model (see also~\cite{Bin16, liu2015posterior, stinis2012stochastic}).  
Related Bayesian approaches to
particle swarm optimization appears
in~\cite{ji2008particle, monson2004kalman}.  

The parametric version of the model-based algorithm is similar
except that at each discrete time-step, the infinite-dimensional
distribution is replaced by its finite-dimensional parametric
approximation.  For example, in the CE algorithm, the parameters are
chosen to minimize the Kullback-Leibler distance (cross-entropy)
between the distribution and its approximation.  In
particle filtering, this is referred to as density projection~\cite{zhou2010POMDP}.         

The two sets of theoretical results in this paper -- the non-parametric
results in \Sec{sec:non-param} and the parametric results in
\Sec{sec:parametric} -- represent the control counterparts of the non-parametric and the
parametric model-based algorithms.  The variational
analysis serves to provide the connection between these as well as
suggest systematic approaches for approximation of the optimal control
law~\eqref{eqn:gradient_gain}.  

Apart from the MEO and PFO algorithms, the non-parametric particle
filter model~\eqref{eqn:particle_dyn} of this paper has some
interesting parallels to the consensus-based optimization
algorithm~\cite{martin2016consensus} where an interacting particle system is
proposed to steer the distribution to the global optimizer.  The
parametric models in this paper are related to the stochastic
approximation type model-based algorithms~\cite{Hu_StoA12} and the natural gradient
algorithm~\cite{kakade}.

\medskip

The outline of the remainder of this paper is as follows: The
variational aspects of the filter -- including the 
non-parametric and parametric cases -- appears in~\Sec{sec:main}. The
details of the two algorithms for numerical approximation of the
control function appear
in~\Sec{sec:control_fn}.  The numerical examples appear
in~\Sec{sec:numerics}.  All the proofs are contained
in the Appendix.

\smallskip

\noindent {\bf Notation:} The Euclidean space $\Re^d$ is equipped
with the Borel $\sigma$-algebra denoted as ${\cal B}(\Re^d)$.  The
space of Borel probability measures on $\Re^d$ with finite second moment is denoted as ${\cal P}$:
\[
{\cal P} \doteq \left\{ \rho:\Re^d\rightarrow [0,\infty) \,
    \text{meas. density} \,\Big{\vert} ~ \int |x|^2 \rho(x) \ud x < \infty \right\}.
\]
The density for a Gaussian random
variable with mean $m$ and variance $\Sigma$ is denoted as
${\cal N}(m,\Sigma)$.  For vectors $x,y\in\Re^d$, the dot product is denoted
as $x\cdot y$ and $|x|:=\sqrt{x\cdot x}$; $x^T$ denotes the transpose
of the vector.  Similarly, for a matrix $\K$, $\K^T$ denotes the
matrix transpose, and $K\succ 0$ denotes positive-definiteness. For
$l,k\in\mathbb{Z}^+$ (Natural numbers), the
tensor notation $\delta_{lk}$ is used to denote the identity matrix
($\delta_{lk}=1$ if $l=k$ and $0$ otherwise).
$C^k$ is used to denote the space of $k$-times
continuously differentiable functions on $\Re^d$. For a function $f$, $\nabla f = \frac{\partial f}{\partial x_i}$ is
used to denote the gradient vector, and 
$D^2 f = \frac{\partial^2 f}{\partial x_i \partial x_j}$ 
is used to denote the Hessian matrix.  $L^\infty$ denotes the
space of bounded functions on $\Re^d$ with associated norm denoted as $\|\cdot\|_{\infty}$. 
$L^2(\Re^d;\pr)$ is the Hilbert space of square integrable functions on
$\Re^d$ equipped with the inner-product, $\big
<\phi,\psi\big>:=\int\phi(x)\psi(x) \pr(x)\ud x$. 
The associated norm is
denoted as $\|\phi\|^2_2:=\big<\phi,\phi\big>$. The space  
$H^1(\Re^d;\pr)$ is the
space of square integrable functions $\phi$ whose derivative (defined in the
weak sense) is in $L^2(\Re^d;\pr)$.  For a function $\phi\in
L^2(\Re^d;\pr)$, $\hat{\phi}:=\int \phi(x) \rho(x) \ud x$ denotes the mean.  $L^2_0$
and $H_0^1$ denote the co-dimension $1$ subspaces of functions whose mean is zero.  

\medskip

\noindent {\bf Assumptions:} 
The following assumptions are made throughout the paper:
\begin{romannum}
\item {\bf Assumption A1:} The prior probability density function
  $p_0^*\in {\cal P}$ and is
of the form $p_0^*(x)= e^{-V_0(x)}$ where $V_0 \in C^2$, $D^2 V_0 \in L^\infty$, and 
\begin{equation*}
\liminf_{|x| \to \infty}\;\;\nabla V_0(x) \cdot \frac{x}{|x|} = \infty.
\end{equation*}
\item {\bf Assumption A2:} The function $h \in C^2 \cap  L^2(\Re^d;p^*_0)$  with $D^2 h \in L^\infty$ and \begin{equation*}
\liminf_{|x| \to \infty}\;\;\nabla h(x) \cdot \frac{x}{|x|} > -\infty.
\end{equation*}
\item {\bf Assumption A3:} 
The function $h$ has a unique minimizer $\bar{x} \in \Re^d$ with
minimum value $h(\bar{x}) =: \bar{h}$. Outside some compact set
$D\subset \Re^d$,  $\exists \; r > 0$ such that
\begin{equation*}
h(x) > \bar{h}+r \quad\quad \forall \; x \in \Re^d \setminus D.
\end{equation*}
\end{romannum}

\medskip

\begin{remark} Assumptions (A1) and (A2) are
important to prove existence, uniqueness and regularity of the
solutions of the Poisson equation (see Appendix~\ref{apdx:Poisson}).
(A1) holds for density with Gaussian tails.    
Assumption (A3) is used to obtain weak convergence of $p^*(x,t)$ to Dirac
delta at $\bar{x}$.  The uniqueness of the minimizer $\bar{x}$ can be
relaxed to obtain weaker conclusions on convergence (See
Appendix~\ref{apdx:exact}).

\end{remark}

\section{Variational Formulation}
\label{sec:main}

\subsection{Non-parametric case}\label{sec:non-param}
A variational formulation of the Bayes
recursion~\eqref{eq:Bayes_with_Yn} is the following {\em time-stepping
  procedure}: For the discrete-time sequence
$\{t_0,t_1, t_2,\hdots,t_{\bN}\}$ with increments $\Delta t_n \doteq
t_n - t_{n-1}$ (see \Sec{sec:intro}),  
set $\rho_0 = p_0^* \in {\cal P}$ and recursively define
$\{\rho_n\}_{n=1}^{\bN} \subset {\cal P}$ by taking $\rho_n \in {\cal
  P}$ to minimize the functional 
\begin{equation}
{\sf I}(\rho|\rho_{n-1})\doteq \frac{1}{\Delta t_n} {\sf D}(\rho \mid \rho_{n-1}) + \beta \int
h(x) \pr(x)\ud x,
\label{eqn:obj_fn}
\end{equation}
where ${\sf D}$ denotes the relative entropy or Kullback--Leibler divergence,
\[
{\sf D}(\rho \mid \rho_{n-1}) :=
\int \rho(x)
\ln \Bigl(\frac{\rho(x)}{\rho_{n-1} (x)} \Bigr) \ud x.
\]

The proof that $\rho_n$, as defined in~\eqref{eq:Bayes_with_Yn}, is in
fact the minimizer is straightforward: By Jensen's formula, ${\sf I}(\rho|\rho_{n-1})
\geq- \ln(\int\rho_{n-1}(y)\exp(- h(y) \Delta t_n)\ud y)$ with equality if
and only if $\rho=\rho_n$.  Although the optimizer is
known,  a careful look at the first order optimality equations
associated with $\rho_n$ leads to i) the replicator dynamics pde for the gradient flow
(in~Theorem~\ref{thm:gradient_flow}), and ii) the proposed particle
filter algorithm for approximation of the posterior
(in~Theorems~\ref{thm:opt-cont} and~\ref{thm:opt-cont-infinite}).  

The sequence of minimizers $\{\rho_n\}_{n=0}^{\bN}$ is used to construct, via
interpolation, a density
function $\rho^{(\bN)}(x,t)$ for $t\in[0,T]$: 
Define $\rho^{(\bN)}(x,t)$ by setting 
\begin{equation*}
\rho^{(\bN)}(x,t) := \rho_n(x),~~\text{for} ~ t \in [t_n, t_{n+1})
\end{equation*}
for $n=0,1,2,\hdots,\bN-1$.  The proof of the following theorem appears in
Appendix~\ref{apdx:evolution}.

\begin{theorem}[Gradient flow]\label{thm:gradient_flow} In the limit as $\bN\bf \to \infty$ the
  density $\pr^{(\bN)}(x,t)$ converges pointwise to the density $\pr(x,t)$
  which  is a weak solution of of the following replicator
    dynamics pde: 
\begin{equation}
 \frac{\partial \rho}{\partial t}(x,t) = -\beta(h(x) - \hat{h}_t)\,\rho(x,t), \quad \pr(x,0)=p_0^*(x).
 \label{eq:replicator}
\end{equation}
\label{thm:evolution}
\end{theorem}

To construct the particle filter, the key idea is to view the gradient flow
time-stepping procedure as a dynamic programming recursion from time
$t_{n-1}\rightarrow t_n$:
\[
\rho_{n} = \mathop{\text{argmin}}_{\rho^{(u)}\in{\cal P}}\;
\underbrace{\frac{1}{\Delta t_n} {\sf D}(\rho^{(u)} |
  \rho_{n-1})}_{\text{control cost}} \; +\;  V(\rho^{(u)}),
\]
where $V(\rho^{(u)}):= \beta \int \rho^{(u)}(x) h (x) \ud x$ is the cost-to-go. 
The notation $\rho^{(u)}$ for density corresponds to the following
construction:  Consider the differential equation
\[
\frac{\ud X_t^i}{\ud t} = u(X_t^i,t)
\]
and denote the associated flow from $t_{n-1}\rightarrow t_n$ as
$x\mapsto s_n(x)$.  Under suitable assumptions on $u$ (Lipschitz in
$x$ and continuous in $t$), the
flow map $s_n$ is a well-defined diffeomorphism on $\Re^d$ and $\rho^{(u)} := s_n^{\#} \left(\rho_{n-1} \right)$, where
$s_n^{\#}$ denotes the push-forward operator.  The push-forward of a
probability density $\rho$ by a smooth map $s$ is defined through the
change-of-variables formula 
\begin{equation*}
	\int f(x) [s^{\#}(\rho)](x) \ud x = \int f(s(x)) \rho(x) \ud x
\end{equation*}
for all continuous and bounded test functions $f$.

Via a formal but straightforward calculation, in the asymptotic limit
as $\Delta t_n \rightarrow 0$, the control cost is
expressed in terms of the control $u$ as
\begin{align}
 \frac{1}{\Delta t_n} {\sf D}(\rho^{(u)} | \rho_{n-1}) &=
  \frac{\Delta t_n}{2} \int \left|\frac{1}{\pr_{n-1}}\nabla \cdot(\pr_{n-1} u) \right|^2
    \,\pr_{n-1}\,\ud x 
    + o(\Delta t_n).
\label{eq:limitinganal}
\end{align}  

\medskip

These considerations help motivate the following optimal control problem:
\begin{equation}
\begin{aligned}
\mathop{\text{Minimize:}}_{u}&\quad J(u) = \int_0^{T} {\sf L}(\pr_t,u_t)\ud t +\beta \int h(x)\pr_T(x)\ud x \\
\text{Constraint:}&\quad\frac{\partial \pr_t}{\partial t} + \nabla
\cdot\left(\pr_t u_t\right) = 0,\quad \pr_0(x)=p_0^*(x), 
\end{aligned}
\label{eq:opt-cont-const}
\end{equation}
where the Lagrangian is defined as
\begin{align*}
{\sf L}(\pr,u) :=  \frac{1}{2} \int_{\Re^d} {\left|\frac{1}{\pr(x)}\nabla \cdot(\pr(x)
  u(x))\right|^2} \;\pr(x)\,\ud x \nonumber \, + \, \frac{\beta^2}{2} \int_{\Re^d} |h(x)-\hat{h}|^2 \;\pr(x)\,\ud x, 
\end{align*} 
where $\hat{h}:=\int h(x) \rho(x)\ud x$.  

The Hamiltonian is defined as
\begin{equation}\label{eq:H_Defn}
{\sf H}(\pr,{\sf q},u) := {\sf L}(\pr,u)- \int {\sf q}(x) \nabla \cdot(\pr(x) u(x)) \,\ud x
\end{equation} 
where ${\sf q}$ is referred to as the momentum.  

\medskip

Suppose $\rho\in{\cal P}$ is the density at time $t$.  The value
function is defined as
\begin{equation}
V(\rho,t) := \inf_{u} \left[\int_t^T {\sf L}(\pr_s,u_s)\ud s\right].
\label{eq:val_fn_def}
\end{equation}
The value function is a functional on the space of
densities.  For a fixed $\rho\in{\cal P}$ and time $t\in[0,T)$, the
(G\^ateaux) derivative of $V$ is a function on $\Re^d$, and an element
of the function space $L^2(\Re^d;\rho)$.  This function is denoted as $\frac{\partial V}{\partial
  \pr}(\pr,t)(x)$ for $x\in\Re^d$.  
Additional details appear in the Appendix~\ref{apdx:opt-cont} where the
following Theorem is proved.

\medskip

\begin{theorem}[Finite-horizon optimal control] 
Consider the optimal control problem~\eqref{eq:opt-cont-const} with
the value function defined in~\eqref{eq:val_fn_def}.  Then $V$ solves
the following DP equation: 
\begin{equation*}
\begin{aligned}
\frac{\partial V}{\partial t}(\pr,t) &+ \inf_{u\in L^2}~ {\sf H}(\pr,\frac{\partial V}{\partial \pr}(\pr,t),u) = 0,\quad t \in [0,T),\\
V(\pr,T) &= \beta \int h(x)\pr(x)\ud x.
\end{aligned}
\end{equation*}
The solution of the DP equation is given by \[V(\rho,t) =\beta \int_{\Re^d} 
h(x) \pr(x) \ud x,\] and the associated optimal control is a solution of the
following pde:
\begin{align}
\frac{1}{\pr(x)}\nabla \cdot(\pr(x) u(x)) 
& = \beta(h(x) - \hat{h}),\quad\forall\;x\in\Re^d. \label{eq:opt-cont-law}
\end{align}
\label{thm:opt-cont}
\end{theorem}

It is also useful to consider the following infinite-horizon version of the optimal control problem:
\begin{equation}
\begin{aligned}
\mathop{\text{Minimize:}}_{u}&\quad J(u) = \int_0^{\infty} {\sf L}(\pr_t,u_t)\ud t  \\
\text{Constraints:}& \quad\begin{cases}\frac{\partial \pr_t}{\partial t} + \nabla
\cdot\left(\pr_t u_t\right) = 0,\quad \pr_0(x)=p_0^*(x), \\
\underset{t \to \infty}{\lim} \int h(x)\pr_t(x) = h(\bar{x}).
\end{cases}
\end{aligned}
\label{eq:opt-cont-const-infinite}
\end{equation}
For this problem, the value function is defined as
\begin{equation}
V(\pr) = \inf_u ~ J(u).
\label{eq:val-func-infinite}
\end{equation}
The solution is given by the following Theorem whose proof appears in
Appendix~\ref{apdx:opt-cont}:

\medskip

\begin{theorem}[Infinite-horizon optimal control]
Consider the infinite horizon optimal control problem
\eqref{eq:opt-cont-const-infinite} with the value function defined
in~\eqref{eq:val-func-infinite}. The value function is given by
\[V(\pr)= \beta\int_{\Re^d}  h(x)\pr(x) \ud x - \beta h(\bar{x})\] 
and the associated optimal control law is a solution of the
pde~\eqref{eq:opt-cont-law}. 
\label{thm:opt-cont-infinite}
\end{theorem} 

\medskip

The particle filter
algorithm~\eqref{eqn:particle_dyn}-\eqref{eqn:gradient_gain} in
Sec.~\ref{sec:intro} is obtained by {\em additionally} requiring the
solution $u$ of~\eqref{eq:opt-cont-law} to be of the gradient form.
One of the advantages of doing so is that the optimizing control law,
obtained instead as solution of~\eqref{eqn:EL_phi_intro}, is uniquely defined
(See Theorem~\ref{thm:Poisson} in Appendix \ref{apdx:Poisson}).  In
part, this choice is guided by the $L^2$ optimality of the gradient
form solution (The proof appears in the Appendix \ref{apdx:opt-cont}):

\medskip

\begin{lemma}[$L^2$ optimality]
Consider the pde~\eqref{eq:opt-cont-law} where $\pr$ and $h$ satisfy
Assumptions (A1)-(A2). 
The general solution is given by
\[
u = - \beta \nabla \phi + v,
\]
where $\phi$ is the solution of~\eqref{eqn:EL_phi_intro}, $v$ solves $\nabla \cdot(\pr v)=0$, and 
\[
\|u\|_2^2  = \beta^2\|\nabla \phi\|^2_2 + \|v\|_2^2.
\]
That is, $u=-\beta\nabla\phi$ is the minimum $L^2$-norm solution
of~\eqref{eq:opt-cont-law}.  
\label{lem:L2norm}
\end{lemma}

\begin{remark}\label{rem:pontmp}
In Appendix~\ref{apdx:Ham}, the Pontryagin's minimum principle of
optimal control is used to express the particle
filter~\eqref{eqn:particle_dyn}-\eqref{eqn:gradient_gain} in its Hamilton's form:
\begin{empheq}[box=\widefbox]{align*}
\frac{\ud X^i_t}{\ud t} &= u(X^i_t,t),\quad X^i_0 \sim p^*_0\\
0 &\equiv  {\sf H}(p(\cdot,t),{\scriptstyle \beta} h, u(\cdot,t)) = \min_{v \in L^2}~ {\sf H}(p(\cdot,t),{\scriptstyle \beta} h, v)
\end{empheq}
The Poisson equation~\eqref{eqn:EL_phi_intro} is simply the first
order optimality condition to obtain a minimizing control.  Under
this optimal control, the density $p(x,t)$ is the optimal trajectory.
The associated optimal trajectory for the momentum (co-state) is a
constant equal to its terminal value $\beta h(x)$.   
\end{remark}

\medskip

The following theorem shows that the particle
filter implements
the Bayes' transport of the density, and establishes the asymptotic
convergence for the density (The proof appears in the Appendix~\eqref{apdx:exact}).  We recall the notation for the two types
of density in our analysis:
\begin{enumerate}
 \item $p(x,t)$:  Defines the density of $X^i_t$.
 \item $p^*(x,t)$: The Bayes' density given by \eqref{eqn:evolution}.
\end{enumerate}

\medskip

\begin{theorem}[Bayes' exactness and convergence]  Consider the particle
filter~\eqref{eqn:particle_dyn}-\eqref{eqn:gradient_gain}.  If
$p(\varble,0)=p^*(\varble ,0)$, we have for all $t\geq 0$,
$$
p(\varble,t) = p^*(\varble,t).
$$
As $t\rightarrow\infty$, 
$\int h(x) p(x,t)\ud x$ decreases monotonically to $h(\bar{x})$ and 
$X_t^i\rightarrow \bar{x}$ in probability.
\label{thm:exact}
\end{theorem}  

\medskip

The hard part of implementing the particle filter is solving the Poisson
equation~\eqref{eqn:EL_phi_intro}.  For the quadratic Gaussian case -- 
where the objective function $h$ is quadratic and the prior $p_0^*$ is
Gaussian -- the solution can be obtained in an
explicit form.  This is the subject of the \Sec{sec:gaussian}.  In the
quadratic Gaussian case, the infinite-dimensional particle filter can
be replaced by a finite-dimensional filter
involving only the mean and the variance of the Gaussian density.  The
simplification arises because the density admits a parameterized form.
A more general version of this result -- finite-dimensional filters for
general class of parametrized densities --  is the subject of
\Sec{sec:parametric}.  For
the general case where a parametric form of density is not
available, numerical algorithms for approximating the control function
solution appear in \Sec{sec:control_fn}.  

\medskip

\begin{remark}
In the construction of the time-stepping procedure~\eqref{eqn:obj_fn},
we considered a gradient flow with respect to the divergence metric.  In
the optimal transportation literature, the Wasserstein metric is
widely used. In the conference version of this paper~\cite{Chi_CDC13}, it is shown that
the limiting density with the Wasserstein metric evolves according to
the Liouville equation:
\[   
\frac{\partial \rho}{\partial t}(x,t) = \nabla\cdot(\rho(x,t)\nabla
   h(x)).
\]
The particle filter is the gradient descent algorithm:
$$\frac{\ud X^i_t}{\ud t} = -\nabla h(X_t^i).$$  
The divergence metric is chosen here because of the Bayesian nature of
the resulting solution. 
\end{remark}

\subsection{Quadratic Gaussian case}
\label{sec:gaussian}

For the quadratic Gaussian problem, the solution of the Poisson equation
can be obtained in an explicit form as described in the following Lemma.  The
proof appears in the Appendix~\ref{apdx:quadratic-Gaussian}. 

\medskip

\begin{lemma}
Consider the Poisson equation~\eqref{eqn:EL_phi_intro}.  Suppose the objective
function $h$ is a quadratic function such that $h(x)\rightarrow \infty$
as $|x|\rightarrow \infty$ and the density $\rho$ is a Gaussian
with mean $m$ and variance $\Sigma$.  Then the control function
\begin{equation}
u(x) = -\beta\nabla\phi(x) = -\beta{\sf K} ( x - m) - \beta b,
\label{eqn:particle_filter_lin}
\end{equation}
where the affine constant vector
\begin{equation}
b =  \int x (h(x) - \hat{h}) \rho(x) \ud x,\label{eq:b-integral}
\end{equation}
and the gain matrix ${\sf K}={\sf K}^T\succ 0$ is the solution of the Lyapunov equation:
  \begin{align}
   \Sigma {\sf K} + {\sf K} \Sigma  = \int (x - m) (x-m)^T (h(x) -
   \hat{h}) \rho(x) \ud x.\label{eq:K-integral}
  \end{align}
  \label{lem:u-lin}
\end{lemma}

Using an affine control law~\eqref{eqn:particle_filter_lin}, it is
straightforward to verify that $p(x,t)=p^*(x,t)$ is a Gaussian whose
mean $m_t\rightarrow \bar{x}$ and variance $\Sigma_t\rightarrow 0$.  The proofs
of the following Proposition and the Corollary appear in the Appendix~\ref{apdx:quadratic-Gaussian}:   

\medskip

\begin{proposition}
Consider the particle filter~\eqref{eqn:particle_dyn} with the affine control
law~\eqref{eqn:particle_filter_lin}.  Suppose the objective
function $h$ is a quadratic function such that $h(x)\rightarrow \infty$
as $|x|\rightarrow \infty$ and the prior density $p_0^*$ is a Gaussian
with mean $m_0$ and variance $\Sigma_0$.  Then the posterior density
$p$ is a Gaussian whose mean $m_t$ and variance $\Sigma_t$ evolve
according to
\begin{equation}
\begin{aligned}
   \frac{\ud m_t}{\ud t} &= -\beta {\sf E} \left[ X_t^i (h(X_t^i) - \hat{h}_t) \right], \\
   \frac{\ud \Sigma_t}{\ud t} &= -\beta{\sf E} \left[ (X_t^i - m_t) (X_t^i-m_t)^T (h(X_t^i) -
   \hat{h}_t)\right], 
\end{aligned}
\label{eq:Gaussian-pf}
\end{equation}
where $\hat{h}_t:= {\sf E}[ h(X_t^i)]$. 
\label{prop:Gaussian}
\end{proposition}

\medskip

\begin{corollary}\label{cor:Gaussian}
Under the hypothesis of Proposition~\ref{prop:Gaussian}, with an
explicit form for quadratic objective function $h(x) =
\frac{1}{2}\,(x-\bar{x})^T H (x-\bar{x}) + c$ where $H = H^T \succ 0
$, the expectations on the righthand-side of~\eqref{eq:Gaussian-pf}
are computed in closed-form and the resulting evolution is given by
\begin{subequations}
\begin{align}
   \frac{\ud m_t}{\ud t} &= \beta \Sigma_t H(\bar{x}-m_t), \label{eq:Gaussian_KF_m}\\
   \frac{\ud \Sigma_t}{\ud t} &= -\beta \Sigma_t H \Sigma_t, \label{eq:Gaussian_KF_s}
\end{align}
\end{subequations}
whose explicit solution is given by
\begin{equation}
\begin{aligned}
   m_t &=  m_0 + \Sigma_0S_t^{-1}(\bar{x}-m_0), \\
   \Sigma_t &= \Sigma_0 - \Sigma_0S_t^{-1}\Sigma_0,
\end{aligned}
\label{eq:exact_soln_QG}
\end{equation}
where $S_t := \frac{1}{\beta t}H^{-1}+\Sigma_0$.  In particular,
$m_t\rightarrow \bar{x}$ and $\Sigma_t\rightarrow 0$. 
\end{corollary}

\medskip

In practice, the affine control law~\eqref{eqn:particle_filter_lin} is implemented as:
  \begin{equation}
   \frac{\ud X^i_t }{\ud t} = -\beta{\sf K}_t^{(N)} (X^i_t-m_t^{(N)}) -
   \beta b_t^{(N)} =: u_t^{i},
   \label{eqn:particle_filter_lin_implement}
  \end{equation}
where the terms are approximated empirically from the ensemble
$\{X_t^i\}_{i=1}^N$. The algorithm appears in 
Table~\ref{alg:affine} (the dependence on time $t$ is suppressed). 

As $N\rightarrow \infty$, the approximations become exact 
and~\eqref{eqn:particle_filter_lin} represents the mean-field limit of
the finite-$N$ control in~\eqref{eqn:particle_filter_lin_implement}.
Consequently, the
empirical distribution of the ensemble approximates the
posterior distribution (density) $p^{*}(x,t)$.

\medskip

\begin{remark}
The finite-dimensional system~\eqref{eq:Gaussian-pf} is the
optimization counterpart of the
Kalman filter.    
Likewise the particle filter~\eqref{eqn:particle_filter_lin_implement}
is the counterpart of the ensemble Kalman filter.  
While the affine control law~\eqref{eqn:particle_filter_lin} is
optimal for the quadratic Gaussian case, it can be implemented for more
general non-quadratic non-Gaussian settings - as long as the various
approximations can be obtained at each step.  The situation is
analogous to the filtering setup where the Kalman filter is often used as
an approximate algorithm even in nonlinear non-Gaussian settings. 
      
\end{remark}

\begin{algorithm}
 \caption{Affine approximation of the control function}
 \begin{algorithmic}[1]
     \REQUIRE $\{X^i\}_{i=1}^N$, $\{h(X^i)\}_{i=1}^N$, $\beta$ 
     \ENSURE $\{u^i\}_{i=1}^N$ \medskip
     \STATE Calculate $m^{(N)} := \frac{1}{N} \sum_{i=1}^N X^i$, 
     \STATE Calculate $\Sigma^{(N)} := \frac{1}{N} \sum_{i=1}^N \left(X^i - m^{(N)}\right) \left(X^i - m^{(N)}\right)^T$
     \STATE Calculate $\hat{h}^{(N)}:= \frac{1}{N}\sum_{i=1}^N h(X^i)$
     \STATE  Calculate $b^{(N)} := \frac{1}{N} \sum_{i=1}^N X^i \; \left( h(X^i) - \hat{h}^{(N)} \right)$
     \STATE Calculate \[C^{(N)}:=\frac{1}{N} \sum_{i=1}^N \left( X^i -m^{(N)} \right) 
     \left( X^i -m^{(N)} \right)^T
     \left( h(X^i)-\hat{h}^{(N)} \right)\]
     \STATE Calculate  ${\sf K}^{(N)}$ by solving~ $\Sigma^{(N)} {\sf K}^{(N)} + {\sf K}^{(N)} \Sigma^{(N)}=C^{(N)}$
     \STATE  Calculate  $u^i=-\beta{\sf K}^{(N)} (X^i-m^{(N)}) - \beta b^{(N)}$
 \end{algorithmic}
 \label{alg:affine}
\end{algorithm}

\subsection{Parametric case}
\label{sec:parametric}
Consider next the case where the density has a known parametric form,
\begin{equation}
p(x,t) = \Varrho(x;\theta_t),
\label{eq:p-param}
\end{equation}
where $\theta_t\in \Re^M$ is the parameter vector.  For example, in
the quadratic Gaussian problem, $\Varrho$ is a Gaussian with
parameters $m_t$ and $\Sigma_t$.   

For the parametric density $\Varrho(x;\vartheta)$, $\frac{\partial }{\partial \vartheta} \left(\log
\Varrho(x;\vartheta)\right)$ is a $M\times 1$ column vector whose
$k^{\text{th}}$ entry,
\[
\left[ \frac{\partial }{\partial \vartheta} (\log
\Varrho(x;\vartheta)) \right]_{k} = \frac{\partial }{\partial \vartheta_k} \left(\log
\Varrho(x;\vartheta)\right),
\]
for $k=1,\hdots,M$.

The Fisher information matrix is a $M\times M$ matrix:
\begin{equation}
G_{(\vartheta)} := \int\frac{\partial }{\partial \vartheta} (\log
\Varrho(x;\vartheta)) \; \left[\frac{\partial }{\partial \vartheta} (\log
\Varrho(x;\vartheta))\right]^T \Varrho(x;\vartheta)\ud x.
\label{eq:Fisher}
\end{equation}
By construction, $G_{(\vartheta)}$ is symmetric and positive
semidefinite.  In the following, it is furthermore assumed that $G_{
(\vartheta)}$ is strictly positive definite, and thus invertible, for all
$\vartheta\in \Re^M$.  

In terms of the parameter, 
\begin{equation*}
e(\vartheta):=\int h(x)\Varrho(x;\vartheta)\ud x,
\end{equation*}
and its gradient is a $M\times 1$ column vector:
\begin{equation}
\nabla e (\vartheta) = \int h(x)\frac{\partial}{\partial \vartheta} \left(\log\Varrho(x;\vartheta)\right)\Varrho(x;\vartheta)\ud x.
\label{eq:grad-e}
\end{equation}

We are now prepared to describe the induced evolution for the parameter vector
$\theta_t$. The proof of the following proposition appears in the
Appendix~\ref{apdx:parametric}. 

\medskip

\begin{proposition}
Consider the particle filter
\eqref{eqn:particle_dyn}-\eqref{eqn:gradient_gain}.  Suppose the
density admits the parametric form~\eqref{eq:p-param} whose Fisher
information matrix, defined in~\eqref{eq:Fisher}, is assumed to be 
invertible.  Then the parameter vector $\theta_t$ is a solution of the
following ordinary differential equation,
\begin{equation}
\frac{\ud\theta_t}{\ud t} = -\beta G^{-1}_{(\theta_t)}\nabla e(\theta_t) .
\label{eq:natural-grad}
\end{equation}
\label{prop:natural-grad}
\end{proposition}

\medskip

\begin{remark}
The filter~\eqref{eq:natural-grad} is referred to as the {\em natural
gradient}; cf.,~\cite{kakade}.  There are several variational
interpretations:
\begin{romannum}
\item The filter can be obtained via a time stepping procedure,
  analogous to~\eqref{eqn:obj_fn}.  
  The sequence $\{\theta_n\}_{n=1}^N$ is inductively defined as a
  minimizer of the function,
\begin{equation*}
{\sf I}(\theta|\theta_{n-1}):=~\left[\frac{1}{\Delta t_n}{\sf
    D}(\Varrho(\cdot;\theta)|\Varrho(\cdot;\theta_{n-1}))+ \beta e(\theta)\right].
\end{equation*} 
On taking the limit as $\Delta t_n \to 0$, one arrives at the
filter~\eqref{eq:natural-grad}. 
\item The optimal control interpretation of~\eqref{eq:natural-grad} is
  based on the
  Pontryagin's minimum principle (see also Remark~\ref{rem:pontmp}).  For
  the finite-dimensional problem, the Hamiltonian
\[
{\sf H}(\theta,{\sf q},u) = {\sf L} (\theta,u) + {\sf q}\cdot u,
\]  
where ${\sf q}\in\Re^M$ is the momentum.  With $\dot{\theta}=u$,
the counterpart of~\eqref{eq:limitinganal} is
\begin{align*}
 \frac{1}{\Delta t_n}{\sf
   D}(\Varrho^{(u)}(\cdot;\theta)|\Varrho(\cdot;\theta_{n-1})) &=\frac{1}{2} u^T G_{(\theta)} u
    + o(\Delta t_n).
\end{align*}  
With $\frac{1}{2} \; u^T G_{(\theta)} u$ as the control cost component
in the Lagrangian, the
first order optimality condition gives
\[
G_{(\theta)} u = - {\sf q} = - \beta \nabla e (\theta),
\]
where we have used the fact that $\beta e(\theta)$ is the value
function.  
Note that it was not necessary to write the explicit form of the
Lagrangian to obtain the optimal control.  
\item Finally, the filter~\eqref{eq:natural-grad} represents the
  gradient flow (in $\Re^M$) for the objective function $e(\theta)$ with
  respect to the Riemannian metric
  $\big<v,w\big>_{\theta}=v^TG_{(\theta)}w$ for all $v,w \in \Re^M$. 
\end{romannum}

\end{remark}

\medskip
\begin{example}
In the quadratic Gaussian case, the natural gradient
algorithm~\eqref{eq:natural-grad} with parameters $m_t$ and $\Sigma_t$ reduces to~\eqref{eq:Gaussian-pf}. 
\label{ex:param-Gaussian}
\end{example}

\medskip

\begin{remark}
While the systems~\eqref{eq:natural-grad} and~\eqref{eq:Gaussian-pf}
are finite-dimensional, the righthand-sides will still need to be
approximated empirically.  The convergence properties of a class of
related algorithms is studied using a stochastic approximation
framework in~\cite{Hu_StoA12}. 
\end{remark}

\medskip

The stochastic approximation is not necessary if the problem admits a
certain affine structure in the parameters: 

\begin{example}\label{ex:param-alpha}
Suppose the density is of the following exponential parametric form:
\begin{equation*}
\Varrho(x;\vartheta) = \frac{\exp( \vartheta \cdot \psi(x))}{\int  \exp(
  \vartheta \cdot \psi(y)) \ud y},
\end{equation*}
where $\vartheta\in\Re^M$, and
$\psi(x):=(\psi_1(x),\psi_2(x),\hdots,\psi_M(x))$ is a {\em given}
set of linearly independent (basis) functions, expressed here as a
vector. Furthermore, suppose $h$ is expressed
as a linear combination of these functions:
\begin{equation*}
h(x) = \alpha \cdot \psi(x),
\end{equation*}
where $\alpha \in \Re^M$. 

The elements of the Fisher information matrix~\eqref{eq:Fisher} and the
gradient~\eqref{eq:grad-e} are given by the respective formulae:
\begin{align*}
[G]_{lk}(\theta)= \int
(\psi_l(x)-\hat{\psi}_l)(\psi_k(x)-\hat{\psi}_k)\Varrho(x;\theta) \ud
x,\\
[\nabla e]_k(\theta) = \int (\alpha\cdot\psi(x)) \; (\psi_k(x)-\hat{\psi}_k)\Varrho(x;\theta)\ud x,
\end{align*}
where $\hat{\psi}_k:=\int \psi_k(x)\Varrho(x;\theta)\ud x$.  The
ode~\eqref{eq:natural-grad} simplifies to
\begin{equation*}
\frac{\ud \theta_t}{\ud t} = -\beta \alpha.
\end{equation*}
\end{example}
Although interesting, there do not appear to be any non-trivial
examples where the affine structure applies.

\section{Numerical Approximation of Control Function}
\label{sec:control_fn}

The Poisson equation~\eqref{eqn:EL_phi_intro} is expressed as
\begin{flalign}
\label{eqn:EL_phi_prelim_vn}
\begin{aligned}
-\Delta_\rho \phi & = h-\hat{h},\\
\int \phi \rho \ud x & = 0,
\end{aligned} &
\end{flalign}
where $\rho\in{\cal P}$ and $\Delta_\rho
\phi:=\frac{1}{\rho}\nabla \cdot (\rho\nabla \phi)$. 
The equation is solved for each time to obtain 
the control function $u(x):=-\beta \nabla\phi(x)$.  The existence-uniqueness theory
for the solution $\phi$ is summarized in Appendix~\ref{apdx:Poisson}.  

\medskip

\noindent \textbf{Problem statement:} Given $N$ 
samples $\{X^1,\hdots,X^i,\hdots,X^N\}$ drawn i.i.d. from $\rho$, approximate
the vector-valued control input $\{u^1,\hdots,u^i,\hdots,u^N\}$, where
$u^i:=u(X^i)=-\beta\,\nabla\phi(X^i)$.  
The density $\rho$ is not explicitly known.  

\subsection{Galerkin Approximation}
\label{sec:Galerkin}

The Galerkin approximation is based upon the weak
form of the Poisson equation~\eqref{eqn:EL_phi_prelim_vn} (see~\eqref{eq:Poisson-weak} in Appendix~\ref{apdx:Poisson}). 
Using the notation $<\cdot,\cdot>$ for
the inner product in
$L^2(\Re^d;\rho)$, the weak form is
succinctly expressed as follows: Obtain $\phi\in H_0^1(\Re^d;\rho)$ such that
\[
\big<\nabla \phi, \nabla \psi \big> = \big<h-\hat{h},\psi\big>,\quad
\forall \; \psi \in H^1(\Re^d;\rho).
\]
The Galerkin approximation 
involves solving this equation in a finite-dimensional subspace
$S:=\text{span}\{\psi_1,\ldots,\psi_M\} \subset H^1_0$, where
$(\psi_1(x),\psi_2(x),\hdots,\psi_M(x))=:\psi(x)$ is a {\em given}
set of linearly independent (basis) functions, expressed as a vector. 
The solution $\phi$ is approximated as
\begin{equation*}
\phiM(x) = c \cdot \psi(x),
\end{equation*}
where the vector
$c \in \Re^M$ is selected such that 
\begin{equation}
\big<\nabla \phiM, \nabla \psi \big> = \big<h-\hat{h},\psi\big>,\quad
\forall \; \psi \in S.
\label{eq:Poisson-weak-finite}
\end{equation}

The finite dimensional approximation~\eqref{eq:Poisson-weak-finite} is
a linear matrix equation
\begin{equation}
Ac=b,
\label{eq:Acb}
\end{equation}
where $A$ is a $M\times M$ matrix and $b$ is a $M\times 1$ vector
whose entries are given by the respective formulae
\begin{align*}
[A]_{lk} & = \big<\nabla\psi_l,\nabla
\psi_k\big>,\\
[b]_l & = \big<h-\hat{h},\psi_l\big>.
\end{align*}

It is next shown that the affine control
law~\eqref{eqn:particle_filter_lin}, introduced in
Lemma~\ref{lem:u-lin} as an exact solution for the quadratic Gaussian
case, is in fact a Galerkin solution.

\medskip

\begin{example} Two types of approximations follow from consideration
  of first order and second order polynomials as basis functions:
\begin{romannum}
\item The {\em constant approximation} is obtained by taking basis
  functions as $\psi_l(x)=x_l$ for $l=1,\ldots,d$.  With this choice, $A$ is the identity matrix
  and the control function is a constant vector:
\begin{equation*}
u(x)=-\beta b = -\beta \int x\, (h(x)-\hat{h})\, \pr(x)\ud x.
\end{equation*}

\item The {\em affine approximation} is obtained by taking the basis
  functions from the family of quadratic polynomials, $\psi_l(x)=x_l$
  for $l=1,\ldots,d$ and $\psi_{lk}(x)= x_{l} \, x_{k}$ for
  $1\le l \le k \le d$.  In this case,
\begin{equation*}
u(x) = -\beta \K(x-m)- \beta b,
\end{equation*}
where the vector $-b$ is the constant approximation, the matrix $\K$
is the solution of the Lyapunov equation~\eqref{eq:K-integral} and
$m:=\int x \, \rho(x) \ud x$ is the mean.  The calculation is included
as part of the proof of Lemma~\ref{lem:u-lin} given in Appendix~\ref{apdx:quadratic-Gaussian}.  
Note that the Galerkin derivation of the affine control
law does not require that the density be Gaussian.  
\end{romannum}
\end{example}

\medskip

\begin{algorithm}[t]
 \caption{Galerkin approximation of control function}
 \begin{algorithmic}[1]
     \REQUIRE $\{X^i\}_{i=1}^N$, $\{h(X^i)\}_{i=1}^N$, $\{\psi_1,\ldots,\psi_M\}$, $\beta$
     \ENSURE $\{u^i\}_{i=1}^N$, \medskip
     \STATE $[A^{(N)}]_{lk} :=\frac{1}{N}\sum_{i=1}^N \nabla \psi_l(X^i) \cdot \nabla \psi_k(X^i),$ for $l,k=1,\ldots,M$
     \STATE $[b^{(N)}]_k :=\frac{1}{N}\sum_{i=1}^N
     \psi_k(X^i)(h(X^i)-\hat{h}^{(N)})$, for $k=1,\ldots,M$ 
     \STATE Calculate $c^{(N)}$ by solving $A^{(N)} c^{(N)}= b^{(N)}$
     \STATE $u^i= -\beta \sum_{k=1}^M c^{(N)}_k \nabla \psi_k(X^i)$
 \end{algorithmic}
 \label{alg:galerkin}
\end{algorithm}

In practice,  the matrix $A$ and the vector $b$ are approximated
empirically, and the equation~\eqref{eq:Acb} solved to obtain the
empirical approximation of $c$, denoted as $c^{(N)}$ (see
Table~\ref{alg:galerkin} for the Galerkin algorithm).  In terms of
this empirical approximation, the control function is approximated as
\[
u(x) = -\beta\nabla\phi^{(M,N)}(x) := -\beta c^{(N)} \cdot \nabla \psi(x).
\]

\begin{example}
With a single basis function $\psi(x) = h(x)$, the approximate Galerkin solution is
\[
\phi(x) = \frac{\int (h(x) - \hat{h})^2 \rho(x) \ud x}{\int |\nabla
  h(x) |^2 \rho(x) \ud x} \; h(x).
\]
Using an empirical approximation, the finite-$N$ system is the gradient-descent algorithm:
\[
\frac{\ud X^i_t }{\ud t} = -\beta \frac{\sum_{i=1}^N (h(X_t^i) -
  \hat{h}^{(N)})^2}{\sum_{i=1}^N |\nabla h(X_t^i) |^2}\; \nabla h(X_t^i).
\] 
\end{example}

\medskip

The following Proposition
provides error bounds for the special case where the basis functions
are chosen to be the eigenfunctions of the Laplacian $\Delta_{\rho}$. 
The proof appears in Appendix~\ref{apdx:Galerkin}.

\medskip

\begin{proposition} Consider the empirical Galerkin approximation of the Poisson
  equation~\eqref{eqn:EL_phi_prelim_vn} on the space
  $S:=\text{span}\{e_1,e_2,\hdots,e_M\}$ of the first $M$
  eigenfunctions of $\Delta_{\rho}$.  Fix $M<\infty$.  Then there
  exists a unique solution for the matrix equation~\eqref{eq:Acb}.
  And there is a sequence of random variables $\{\epsilon_N\}$ such that  
\begin{equation*}
\|\nabla\phi-\nabla\phiMN \|_2  \leq  \frac{1}{\sqrt{\lambda_M}}\|h-\Pi_S h\|_2+ \epsilon_N,
\end{equation*}
where $\epsilon_N\rightarrow 0$ as $N\rightarrow\infty$ a.s, and
$
\Pi_S h :=  \sum_{k=1}^M <e_k,h>e_k
$ is the projection of the function $h$ onto $S$. 
\label{prop:galerkin}
\end{proposition}

\medskip

\begin{remark}[Variational interpretation] Suppose $u=-\beta\nabla\phi$ is
  the exact control function obtained from solving the weak form of
  the Poisson equation~\eqref{eqn:EL_phi_prelim_vn}.  
The Galerkin solution $\nabla\phiM$ is the optimal least-square
approximation of $\nabla\phi$ in $S\subset H^1(\Re^d;\rho)$, i.e, 
 \begin{equation*}
 \phiM =\argmin_{\psi \in S} ~\|\nabla \phi-\nabla \psi\|_2.
 \end{equation*}
The Galerkin approximation~\eqref{eq:Poisson-weak-finite} is simply
the statement of the projection theorem. 
\end{remark}

\begin{algorithm}[h]
 \caption{Kernel-based approximation of control function}
 \begin{algorithmic}[1]
     \REQUIRE $\{X^i\}_{i=1}^N$, $\{h(X^i)\}_{i=1}^N$, $\Phi_{\text{prev}}$,
     $\beta$ , $L$
     \ENSURE $\{u^i\}_{i=1}^N$ \medskip
     \STATE Calculate $g_{ij}:=\exp(-|X^i-X^j|^2/4\epsilon)$ for $i,j=1$ to $N$\medskip
     \STATE Calculate $k_{ij}:=\frac{g_{ij}}{\sqrt{\sum_l g_{il}}\sqrt{\sum_l g_{jl}}}$ for $i,j=1$ to $N$
     \STATE Calculate $T_{ij}:=\frac{k_{ij}}{\sum_l k_{il}}$ for $i,j=1$ to $N$
     \STATE Calculate $\hat{h}^{(N)}=\frac{1}{N}\sum_{i=1}^N h(X^i)$\medskip
     \STATE Initialize $\Phi_i=\Phi_{\text{prev},i}$ for $i=1$ to $N$
     \medskip
     \FOR {$l=1$ to  $L$}
     \STATE Calculate $\Phi_i= \sum_{j=1}^N T_{ij} \Phi_j + \epsilon (h(X^i)-\hat{h}^{(N)})$\medskip
     \STATE Calculate $\Phi_i = \Phi_i - \frac{1}{N}\sum_{j=1}^N \Phi_j$
     \ENDFOR
     \STATE Calculate 
      \[
      u^i = \frac{-\beta}{2\epsilon}\sum_{j=1}^N \left[T_{ij}(\Phi_j + \epsilon(h(X^j) 
      - \hat{h}^{(N)}))\left(X^j - \sum_{k=1}^N T_{ik}X^k\right)\right]
      \]
 \end{algorithmic}
 \label{alg:kernel}
\end{algorithm}

\subsection{Kernel-based approximation}
\label{sec:kernel}

An alternate algorithm 
is based on approximating the semigroup of $\Delta_{\rho}$.  The semigroup, introduced in
Appendix~\ref{apdx:Poisson}, is denoted as $e^{\epsilon\Delta_\rho}$. The solution $\phi$ of the Poisson
equation~\eqref{eqn:EL_phi_prelim_vn} is equivalently expressed as,
for any fixed $\epsilon>0$,
\begin{equation}
\phi = e^{\epsilon\Delta_\rho} \phi+ \int_0^\epsilon e^{s\Delta_\rho} (h-\hat{h}) \ud s.
\label{eq:fixed1_n}
\end{equation}

For the purposes of numerical approximation, $e^{\epsilon \Delta_\rho}$ is
approximated by a finite-rank operator:
\begin{equation}
\TepsN f(x) :=\frac{1}{N}\sum_{i=1}^N \kepsN(x,X^i)f(X^i),
\label{eq:TepsN}
\end{equation}
where the kernel,
\begin{equation*}
\kepsN (x,y) = \frac{1}{\nepsN(x)}\frac{\geps(x-y)}{\sqrt{\frac{1}{N}\sum_{i=1}^N \geps(x-X^i)}\sqrt{\frac{1}{N}\sum_{i=1}^N\geps(y-X^i)}},
\end{equation*}
is expressed in terms of the Gaussian kernel $\geps
(z):={(4\pi\epsilon)}^{-\frac{d}{2}}\exp{(-\frac{|z|^2}{4\epsilon})}$
for $z\in\Re^d$,
and $\nepsN(x)$ is a normalization factor chosen such that $\TepsN
1=1$. It is shown in \cite{coifman,hein2006} that
$e^{\epsilon\Delta\pr}\approx \TepsN$ as $\epsilon \downarrow 0$ and
$N \to \infty$.

The approximation of the fixed-point problem~\eqref{eq:fixed1_n} is
obtained as
\begin{align}
\phiepsN = \TepsN \phiepsN + \epsilon(h-\hat{h}),
\label{eq:fixed-point-epsN}
\end{align}
where $\int_0^\epsilon e^{s\Delta_\pr}(h-\hat{h})\ud s \approx
\epsilon(h-\hat{h})$ for small $\epsilon>0$.   The method of successive approximation is
used to solve the fixed-point equation for
$\phiepsN$.  In a recursive simulation, the method is initialized with
the solution from the previous time-step.

The control function is obtained by taking the gradient of
the two sides of~\eqref{eq:fixed-point-epsN}.  For this purpose, it is
useful to first define a finite-rank operator: 
\begin{align*}
 & \nabla \TepsN f(x) := \frac{1}{N}\sum_{i=1}^N\nabla \kepsN(x,X^i)f(X^i) \\
 & = \frac{1}{2\epsilon}\left[\frac{1}{N}\sum_{i=1}^N \kepsN(x,X^i)f(X^i)\left(X^i - \frac{1}{N}\sum_{j=1}^N\kepsN(x,X^j)X^j\right) \right].
\end{align*}
The control function is approximated as,
\begin{equation*}
u(x) = -\beta \nabla \TepsN \phiepsN(x) -\beta \epsilon\nabla \TepsN(h-\hat{h}^{(N)})(x),
\label{eq:uepsN}
\end{equation*}
where $\phiepsN$ on the righthand-side is the solution
of~\eqref{eq:fixed-point-epsN}.  

The overall algorithm appears in
Table~\ref{alg:kernel}.  The convergence analysis of this algorithm, as $\epsilon\downarrow 0$
and $N\rightarrow \infty$, is outside the scope of this paper and will be published
separately.

\begin{figure}[t]
  \centering
    \includegraphics[scale=0.4]{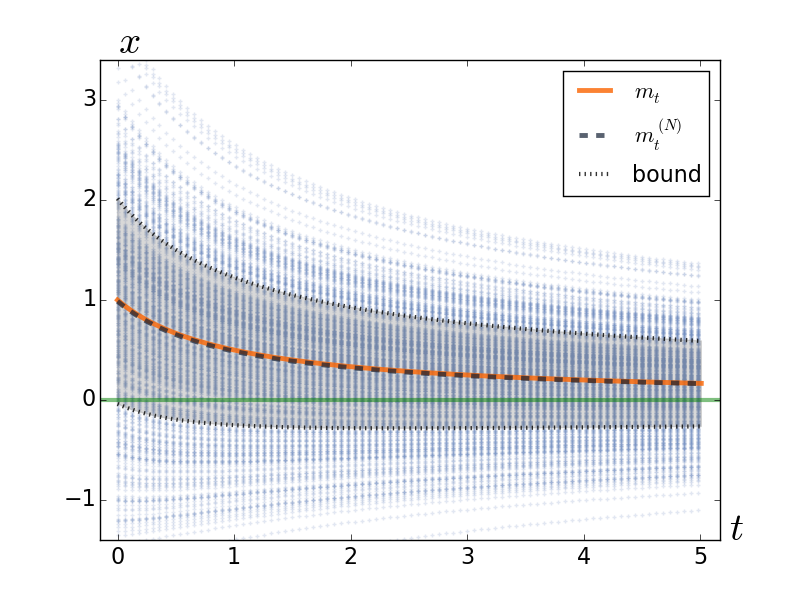}
    \vspace*{-5pt}
  \caption{Simulation results with a quadratic function $h(x) =
    \frac{1}{2}x^2$. Trajectories of $N=500$ particles is depicted as dots in the
    background.  The solid line is the mean $m_t$ obtained using the
    exact formula~\eqref{eq:exact_soln_QG} and the 
    dashed line is its empirical estimate obtained using the
    particles.  The shaded region
    depicts the $\pm 1$ standard deviation bound.}
  \label{fig:Q_convergence}
\end{figure}

\begin{figure*}
  \centering
    \hspace*{-30pt}
    \subfigure[]{
    \includegraphics[scale=0.3]{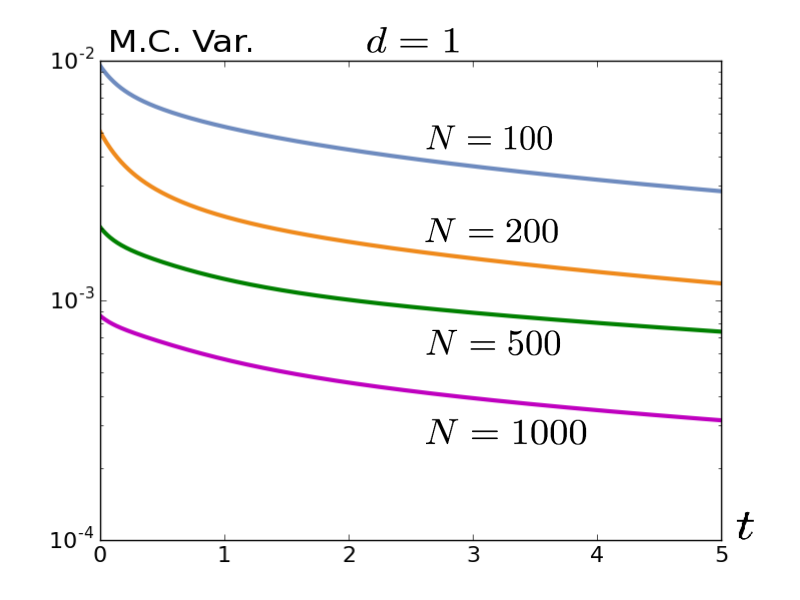}}
    \hspace*{-5pt}
    \subfigure[]{
    \includegraphics[scale=0.3]{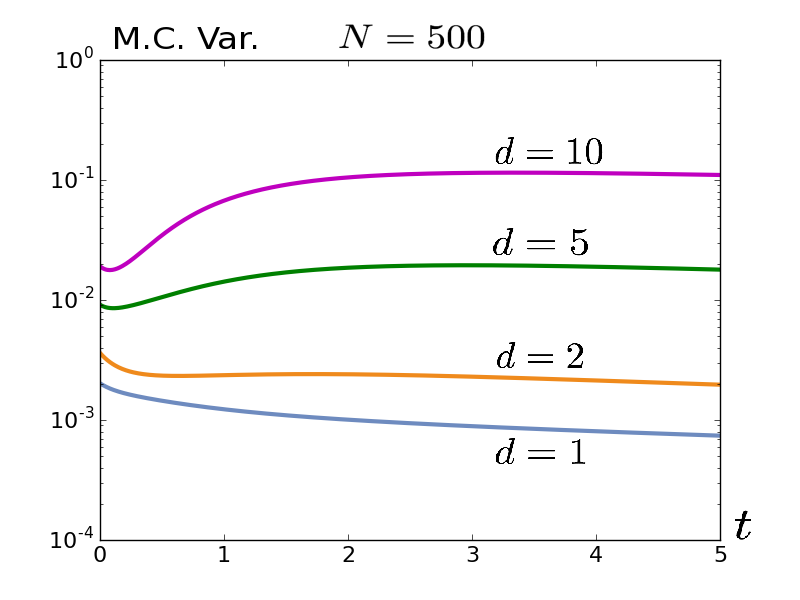}}
    \hspace*{-20pt}
  \caption{Comparison of the M.C. variance of the estimated mean as a function of the number of particles $N$ (part~(a)) and 
  the dimension $d$ (part~(b)).}
  \label{fig:Q_MC_var}
\end{figure*}

\begin{figure}
  \centering
    \includegraphics[scale=0.4]{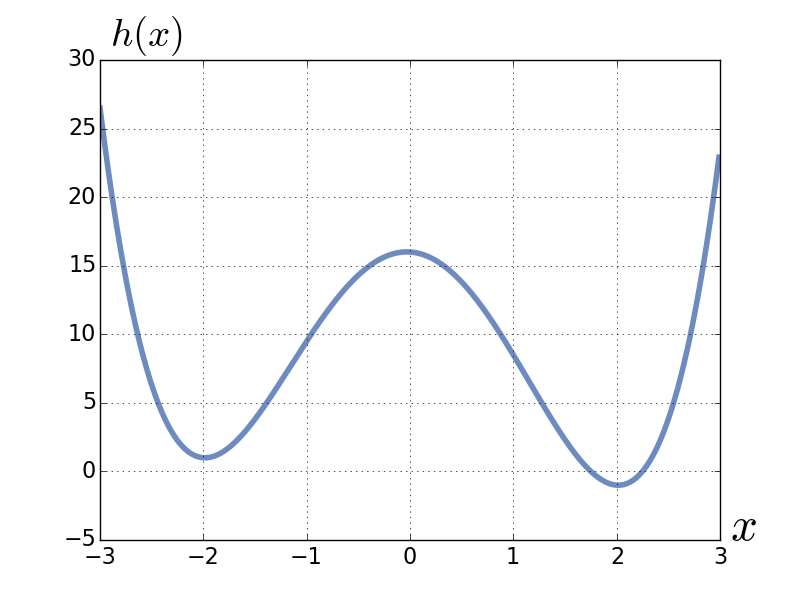}
    \vspace*{-5pt}
  \caption{Double-well potential.}
  \label{fig:DW_fn}
\end{figure}

\section{Numerical Examples} 
\label{sec:numerics}

Results of numerical experiments are presented next.  The purpose is
to illustrate the algorithms with simple examples.  A comprehensive
numerical study on benchmark problems including comparisons with
other algorithms will be a subject of a separate publication.  

\subsection{Quadratic Gaussian Case}

\def\MC{J}

\begin{figure}
  \centering
    \includegraphics[scale=0.32]{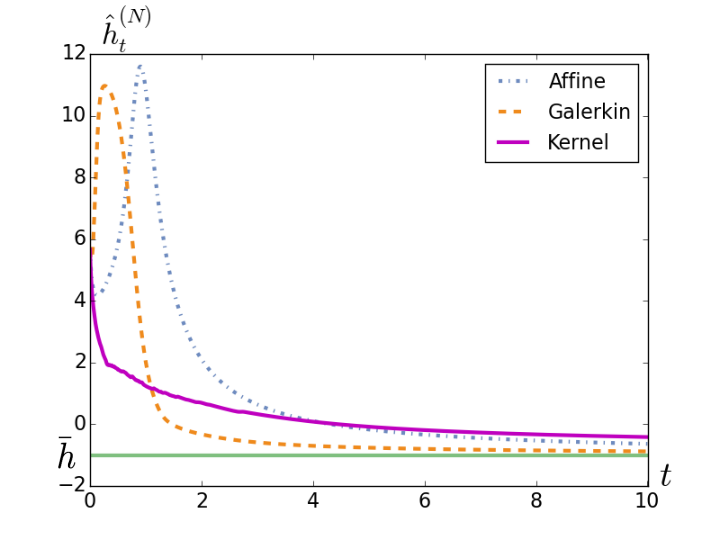}
    \vspace*{-5pt}
  \caption{Comparison of $\hat{h}_t^{(N)}$ with the three types of
    approximate control laws.}
  \label{fig:DW_h_compare}
\end{figure}

Consider the quadratic function $h(x) = \frac{1}{2}\,|x|^2$ for $x
\in \Re^d$.  For a quadratic function with a Gaussian prior, the
optimal solution is known in closed-form (see \Sec{sec:gaussian}).  

The simulation parameters are as follows: The simulations are carried
out over a finite time-horizon $[0,T]$ with $T=5$, a fixed time step
$\Dt=0.01$, and the parameter $\beta=1$.  An Euler discretization is used
to numerically integrate the ode~\eqref{eqn:particle_dyn}. At each discrete time step, the
algorithm in Table~\ref{alg:affine} is used to approximate the affine
control law. The filter is initialized with
samples drawn i.i.d. from the Gaussian distribution $\clN(m_0,
\Sigma_0)$, where $m_0 = (1, ..., 1) \in \Re^d$ and $\Sigma_0 =
\text{diag}(1,...,1)$.

Figure~\ref{fig:Q_convergence} depicts a typical simulation result for $d=1$ and $N=500$.  The empirical mean $m_t^{(N)}$
is seen to closely match its mean-field limit $m_t$ obtained using the exact
formula~\eqref{eq:exact_soln_QG}. 

Figure~\ref{fig:Q_MC_var} depicts the results of Monte Carlo
experiments: The part~(a) is a plot of Monte Carlo (M.C.) variance for
estimated mean as the number of particles $N$ is varied with $d=1$ fixed, and
the part~(b) is the corresponding plot as $d$ is varied with $N=500$
fixed.  
The M.C. variance is defined as:
\begin{align*}
  \text{M. C. Var}(m_t^{(N)})  & := \frac{1}{\MC} \sum_{j=1}^\MC
  |m_{t,j}^{(N)} - \frac{1}{\MC} \sum_{j=1}^{\MC} m_{t,j}^{(N)} |^2
\end{align*}   
with $\MC=100$ independent runs used in the experiment.

\subsection{Double-Well Potential}

Consider a double-well potential $h(x) = (x-2)^2(x+2)^2 -
\frac{x}{2}$, as depicted in Figure~\ref{fig:DW_fn}.  
Figure~\ref{fig:DW_h_compare} depicts a comparison of $\hat{h}_t^{(N)}=N^{-1}\sum_i
h(X_t^i)$ with the three types of approximate control laws: 
the affine approximation given in Table~\ref{alg:affine}, the Galerkin
approximation given in
Table~\ref{alg:galerkin}, and the kernel approximation given in
Table~\ref{alg:kernel}.  With the optimal control,
Theorem~\ref{thm:exact} shows that $\hat{h}_t$ decreases monotonically
as a function of time.  
This was indeed found to be the case with
the kernel-based algorithm but not so with the other two.  Even though
the particles in all three cases eventually converge to the correct
equilibrium (see Fig.~\ref{fig:DW_particle_compare}), 
the approximate nature of the control can lead to a transient growth of $\hat{h}_t^{(N)}$. 

\begin{figure*}
  \centering
    \hspace*{-30pt}
    \subfigure[]{
    \includegraphics[scale=0.30]{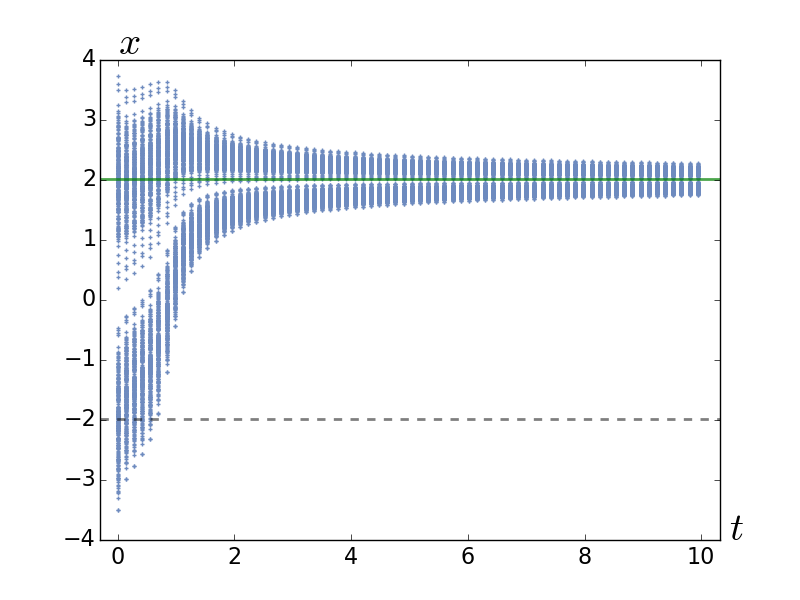}}
    \hspace*{-25pt}
    \subfigure[]{
    \includegraphics[scale=0.30]{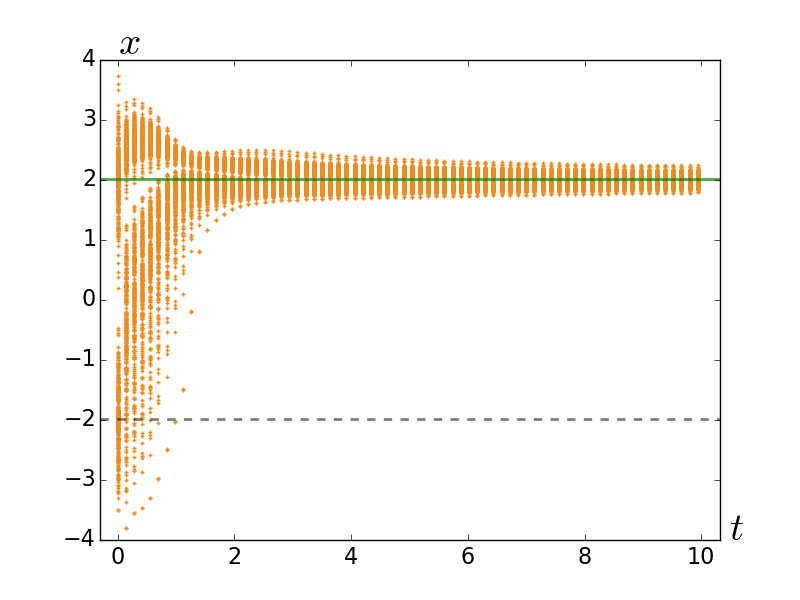}}
    \hspace*{-25pt}
    \subfigure[]{
    \includegraphics[scale=0.30]{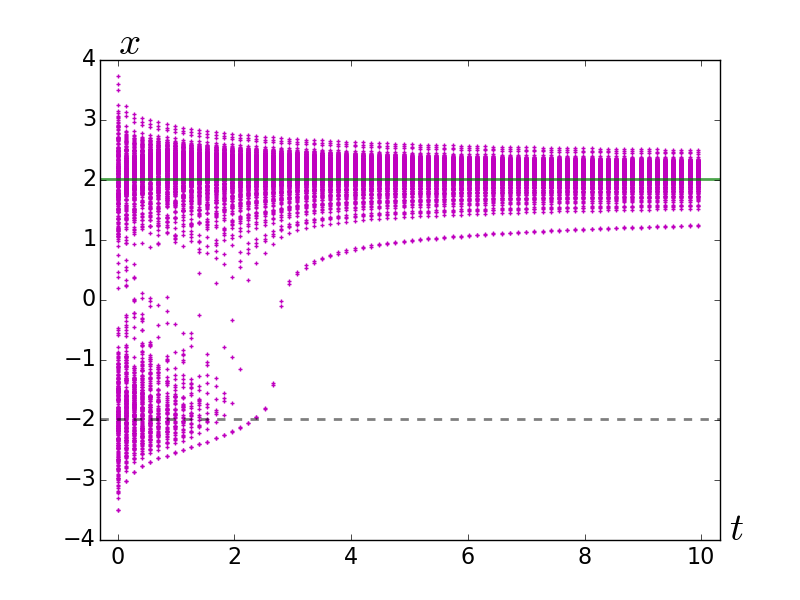}}
    \hspace*{-20pt}
  \caption{Particle trajectories with the three approximate control
    laws: (a) affine control, (b) Galerkin control, and (c)
    kernel-based control.  The global and the local minimizers are
    depicted via the solid and dashed lines, respectively. }
  \label{fig:DW_particle_compare}
\end{figure*}

For each simulation, $N=500$ particles are
used.  The initial particles $X_0^i$ are sampled i.i.d. from a mixture
of two Gaussians, $\clN(-2,\,0.6^2)$ and $\clN(2,\,0.6^2)$, with equal
weights.  An Euler discretization is used to numerically integrate the
ode~\eqref{eqn:particle_dyn} with $\Dt=0.01$ and $\beta=1$.
For the Galerkin approximation, the basis
functions are $ \text{span} \Big\{ x, ~\cos\big( \frac{2\pi}{10}\,x
\big),~ \sin\big( \frac{2\pi}{10}\,x \big) \Big\}$. For the kernel
approximation, the parameter $\epsilon = 0.5$.  The Galerkin
approximation can suffer from numerical instability on account of
ill-conditioning of the matrix $A$.  This can lead to relatively large
values of control requiring small time-steps for numerical integration.

\appendix
\subsection{Poisson's Equation}
\label{apdx:Poisson}

This section includes background on the existence-uniqueness results for
the Poisson equation~\eqref{eqn:EL_phi_intro}.  
The appropriate function space for the solution is the co-dimension 1 subspace 
$L^2_0(\Re^d,\pr):=\{\phi \in L^2(\Re^d,\pr); \int \phi \rho \ud x =0\}$ and $H^1_0(\Re^d,\pr) :=\{\phi \in
H^1(\Re^d,\pr); \int \phi \rho \ud x =0\}$; cf.~\cite{laugesen15,Tao_Automatica}. 

A function $\phi \in H_0^1(\R^d;\rho)$ is said to be a weak solution
of the Poisson's equation~\eqref{eqn:EL_phi_intro} if
\begin{equation}
\int \nabla \phi (x) \cdot \nabla \psi(x) \rho(x) \ud x = \int  (h(x)-\hat{h}) \psi(x) \rho(x) \ud x \label{eq:Poisson-weak}
\end{equation}
for all $\psi \in H^1(\R^d;\rho)$. 

\medskip

\begin{theorem}[Theorem 2.2. in~\cite{laugesen15}]\label{thm:Poisson}
Suppose the density $\rho$ admits a spectral gap (or Poincar\'e
inequality) (\cite{bakry2013} Thm 4.6.3), i.e., $\exists\; \lambda_1
>0$ such that
\begin{equation}
\label{eq:poincare}
\int f(x)^2 \, \rho(x) \ud x \leq  \frac{1}{\lambda_1} \int |\nabla f(x)|^2 \, \rho(x) \ud x ,\quad \forall f \in H^1_0(\Re^d,\pr).
\end{equation}
Then for each $h\in L^2(\R^d;\rho)$, there exists a unique weak solution
$\phi \in H_0^1(\R^d;\rho)$ satisfying
\eqref{eq:Poisson-weak}. Moreover, the derivative of the solution is
controlled by the size of the data:
\begin{equation}
\int |\nabla\phi(x)|^2 \rho(x) \ud x  \le \frac{1}{\lambda_1} \int |h(x) -
\hat{h}|^2 \rho(x) \ud x.\label{eqn:bound1}
\end{equation}
\qed
\end{theorem}

\medskip

An alternate but equivalent approach to obtain the solution
of~\eqref{eqn:EL_phi_intro} is to first note that the weighted
Laplacian $\Delta_\rho :=\frac{1}{\rho}\nabla \cdot (\rho\nabla )$ is
the infinitesimal generator of a Markov semigroup, denoted in this
paper as $e^{\epsilon \Delta_{\rho}}$; cf.,~\cite{bakry2013}.  In
terms of this semigroup, the Poisson's
equation~\eqref{eqn:EL_phi_intro} is equivalently expressed as, for
fixed $\epsilon > 0$,
\begin{equation}
\phi = e^{\epsilon\Delta_\rho} \phi+ \int_0^\epsilon e^{s\Delta_\rho} (h-\hat{h}) \ud s.
\label{eq:fixed1}
\end{equation}
If the density $\rho$ admits a spectral gap (i.e.,~\eqref{eq:poincare}
holds for some $\lambda_1>0$) then $e^{\epsilon\Delta_\rho}$ is a contraction
on $L^2_0(\Re^d,\pr)$ and a unique solution exists by the contracting mapping
theorem. 

The two formulations for obtaining the solution,
viz.,~\eqref{eq:Poisson-weak} and~\eqref{eq:fixed1}, inform the two
algorithms for numerically approximating the control law.  These
algorithms are presented in \Sec{sec:Galerkin} and \Sec{sec:kernel}, respectively.

\subsection{Gradient flow}
\label{apdx:evolution}
\newP{Proof of Theorem \ref{thm:evolution}}
As $\Delta t_n\downarrow0$, $\rho^{(\bN)}(x,t) \rightarrow
p^*(x,t)$, the posterior density defined
in~\eqref{eqn:evolution}. By direct substitution, it is verified
that $p^*$ is a solution of the replicator pde~\eqref{eq:replicator}. 

In the conference version of this paper (see~\cite{Chi_CDC13}), the
replicator pde is derived based on variational analysis.  The main
steps of the variational proof are as follows:
\begin{romannum}
\item By taking the first variation of the
  functional~\eqref{eqn:obj_fn}, the minimizer $\rho_n$ is shown to
  satisfy the E-L equation:
\begin{equation}
\int \frac{\rho_n}{\rho_{n-1}} \nabla\cdot (\rho_{n-1}\,\varsigma)\ud x 
- \Delta t_n\; \beta \int \nabla h\cdot\varsigma \rho_n\ud x = 0,
\label{eq:EL1}
\end{equation}
for each vector field $\varsigma \in L^2(\Re^d;\rho_{n-1})$.
\item Given any $C^1$ smooth and compactly supported (test) function $f$, let
  $\xi_n\in L^2(\Re^d;\rho_{n-1})$ be the solution of
\begin{equation}
\nabla\cdot(\rho_{n-1}\xi_n)=(f-\hat{f}_{n-1})\rho_{n-1},
\label{eqn:equation1}
\end{equation}
where $\hat{f}_{n-1} := \int f \rho_{n-1}\ud x$.  Then, using the E-L
equation~\eqref{eq:EL1},
\[
\hat{f}_n-\hat{f}_{n-1} = \Delta t_n\; \beta  \int \nabla h\cdot\xi_n \rho_n\ud x ,
\]
and upon summing,
\begin{equation}
\hat{f}_{\bN} = \hat{f}_0 + \beta \sum_{n=1}^{\bN} \Delta t_n\int \nabla h\cdot\xi_n \rho_n \ud x.
\label{eqn:sum_up}
\end{equation}
\item Integrating by parts and using~\eqref{eqn:equation1},
\[
\int \nabla h\cdot\xi_n \rho_n \ud x =  -\int
h(f-\hat{f}_{n-1})\rho_{n-1}\ud x + {\cal E}_n,
\]
where the error term ${\cal E}_n = O(\Delta
t_n)$. Equation~\eqref{eqn:sum_up} thus becomes
\[
 \hat{f}_{\bN} = \hat{f}_0 - \beta \sum_{n=1}^{\bN} \Delta
t_n \int h(f-\hat{f}_{n-1})\rho_{n-1}\ud x + \sum_{n=1}^{\bN} \Delta
t_n \mathcal{E}_n .
\]
\item On taking the limit as $\Delta t_n\downarrow0$, the limiting
  density $p^*(x,t)$ satisfies
\begin{align*}
\hat{f}_t & = \hat{f}_0-\beta \int_0^t\int h(x)(f(x)-\hat{f}_s)p^*(x,s) \ud x \ud s \\
& = \hat{f}_0-\beta \int_0^t\int(h(x)-\hat{h}_s)f(x)p^*(x,s) \ud x \ud s
\end{align*}
for all test functions $f$, showing that $p^*(x,t)$ is a weak solution 
of the replicator pde~\eqref{eqn:evolution}.  
For additional details on these calculations, see~\cite{Chi_CDC13}. 
\end{romannum}

\subsection{Optimal control}
\label{apdx:opt-cont}
\newP{Preliminaries} Consider a functional $E:\calP \to \Re$
mapping densities to real numbers.  For a fixed $\rho\in \calP$, the
(G\^ateaux) derivative of $E$ is a real-valued function on $\Re^d$, and an element
of the function space $L^2(\Re^d;\rho)$.  This function is denoted as
$\frac{\partial E}{\partial\pr}(\pr,t)(x)$ for $x\in\Re^d$, and 
defined as follows:  
\begin{equation*}
\frac{\ud}{\ud t}E(\pr_t) \bigg|_{t=0}= -\int_{\Re^d}\frac{\partial E}{\partial \pr}(\pr)(x) \nabla \cdot(\pr(x)u(x))\ud x,
\end{equation*} 
where $\pr_t$ is a path in $\calP$ such that $\frac{\partial
  \pr_t}{\partial t} =- \nabla \cdot(\pr_t u)$ with $\rho_0=\rho$, and $u$
is any arbitrary vector-field on $\Re^d$.  
Similarly, $\frac{\partial^2 E}{\partial \pr^2}(\pr)\in
L^2(\Re^d\times \Re^d)$ is the second (G\^ateaux) derivative of the
functional $E$ if 
\begin{equation*}
\frac{\ud}{\ud t}\frac{\partial E}{\partial \pr}(\pr_t)(x)\bigg|_{t=0} 
= - \int_{\Re^d}\frac{\partial^2 E}{\partial \pr^2}(\pr)(x,y)\nabla \cdot(\pr(y)u(y))\ud y.
\end{equation*}

The optimal control problems~\eqref{eq:opt-cont-const}
and~\eqref{eq:opt-cont-const-infinite} are examples of the mean-field
type control problem introduced in~\cite{bens-master}.  The notation
and the methodology for the following proofs is based in part on~\cite{bens-master}.

\newP{Proof of Theorem~\ref{thm:opt-cont}}
The value function $V(\pr,t)$, defined in \eqref{eq:val_fn_def}, is the solution of the DP equation: 
\begin{equation}
\begin{aligned}
\frac{\partial V}{\partial t}(\pr,t) &+ \inf_{u\in L^2}~ {\sf H}(\pr,\frac{\partial V}{\partial \pr}(\pr,t),u) = 0,\quad t \in [0,T),\\
V(\pr,T) &= \beta \int h(x)\pr(x)\ud x.
\end{aligned}
\label{eq:dyn-prog-apdx}
\end{equation} 
In the following, we use the notation \[\Theta=\Theta(\rho,t)(x):=\frac{\partial V}{\partial
  \pr}(\rho,t)(x).\]  For a fixed $\rho\in\calP$ and
$t\in[0,T)$, $\Theta$ is a function on $\Re^d$.  

A necessary condition is obtained by considering the first variation
of ${\sf H}$. Suppose $u$ is a minimizing control function.
Then $u$ satisfies the first order optimality condition:
\begin{equation*}
\left. \frac{\ud}{\ud \epsilon}{\sf H}(\pr,\Theta,u +\epsilon v) \right|_{\epsilon=0} = 0,
\end{equation*}
where $v$ is an arbitrary vector field on $\Re^d$. 
Explicitly, 
\begin{equation*}
\int \nabla(-\frac{1}{\pr}\nabla \cdot (\pr u) + \Theta ) \cdot v \;\pr \ud x = 0,
\end{equation*}
or in its strong form
\begin{equation*}
-\frac{1}{\pr}\nabla \cdot (\pr u) + \Theta = (\text{constant}).
\end{equation*} 
Multiplying both sides by $\pr$ and integrating yields the
value of the constant as $\int \Theta(\pr,t)(x) \pr(x) \ud x
=:\hat{\Theta}(\pr,t)$.  Therefore, the minimizing control solves the pde
\begin{equation*}
\frac{1}{\pr}\nabla \cdot (\pr u) = \Theta -\hat{\Theta}.
\end{equation*} 
On substituting the optimal control law into the DP
equation~\eqref{eq:dyn-prog-apdx}, the HJB equation for
the value function is given by
\begin{equation*}
\begin{aligned}
\frac{\partial V}{\partial t}(\pr,t) &+\frac{\beta^2}{2} \int
|h-\hat{h}|^2\pr \ud x  \\
& -\frac{1}{2} \int(\Theta(\pr,t)-\hat{\Theta}(\pr,t))^2\pr \ud
x=0,\quad  t \in [0,T),\\
V(\pr,T) &= \beta \int h(x)\pr(x)\ud x.
\end{aligned}
\end{equation*}
The equation involves both $V$ and $\Theta$.  One obtains the
so-called master equation (see~\cite{bens-master}) involving only
$\Theta$ by differentiating with respect to $\pr$
\begin{equation*}
\begin{aligned}
&\frac{\partial \Theta}{\partial t}(\pr,t)(x)+\frac{\beta^2}{2}|h(x)-\hat{h}|^2 -\frac{1}{2}|\Theta(\pr,t)(x)-\hat{\Theta}(\pr,t)|^2 \\ 
&- \int (\Theta(\pr,t)(y)-\hat{\Theta}(\pr,t))\frac{\partial
  \Theta}{\partial \pr}(\pr,t)(y,x)\pr (y)\ud y=0,\; t \in [0,T),\\
&\Theta(\pr,T) = \beta h.
\end{aligned}
\end{equation*}
It is easily verified that $\Theta(\pr,t)=\beta h$ solves
the master equation. The corresponding value function
$V(\pr,t)=\beta \int h \pr \ud x$.

\newP{Sufficiency} The proof that the proposed control law is a
minimizer is as follows.  Consider any arbitrary control law
$v_t$ with the resulting density $\rho_t$. Taking the time derivative of $-\beta \int h\pr_t\ud x$:
\begin{align*}
- \beta \frac{\ud}{\ud t} \int h\pr_t\ud x&= \beta \int h\;\nabla \cdot (\pr_t v_t)\ud x \\
&= \int \beta(h-\hat{h}_t) \; (\frac{1}{\pr_t}\nabla \cdot (\pr_t
v_t))\; \pr_t\ud x \\
& \leq \int \left(\frac{\beta^2}{2}(h-\hat{h})^2 + \frac{1}{2} \left|
    \frac{1}{\pr_t}\nabla \cdot(\pr_t v_t) \right|^2\right) \pr_t \ud x\\ 
& = {\sf L}(\pr_t,v_t).
\end{align*}
On integrating both sides with respect to time,
\begin{align*}
\beta \int_{\Re^d} h\pr_0 \ud x & \;\leq\; \int_0^T{\sf L}(\pr_t,v_t) \ud t\; + \;\beta\int_{\Re^d} h \pr_T \ud x,
\end{align*}
where the equality holds with $v_t=u_t$ (defined as solution
of~\eqref{eq:opt-cont-law}). Therefore,  
\begin{equation*}
J(u)= \beta \int h \pr_0\ud x \leq J(v).
\end{equation*}
This also confirms that $V(\pr,t)=\beta \int h \pr \ud x$ is the value
function, and completes the proof of Theorem~\ref{thm:opt-cont}. \qed

\medskip

The analysis for the infinite horizon optimal control
problem~\eqref{eq:opt-cont-const-infinite} is similar and described next.

\newP{Proof of Theorem \ref{thm:opt-cont-infinite}} 
The infinite-horizon value function $V^{\infty}(\pr) := \mathop{\inf}_{u}
\int_0^\infty {\sf L}(\pr_t,u_t)\ud t$ is a solution of the DP equation:
\begin{equation}
\begin{aligned}
\inf_{u\in L^2}~{\sf H}(\pr,\Theta^{\infty}(\pr),u)= 0,
\end{aligned}
\label{eq:DP_static}
\end{equation}
where $\Theta^{\infty}(\pr):=\frac{\partial V^{\infty}}{\partial
  \pr}(\pr)$.  By carrying out the first order analysis in an
identical manner, it is readily verified
that:
\begin{romannum}
\item A minimizing control $u$ is a solution of the pde~\eqref{eq:opt-cont-law}; 
\item $V^{\infty}(\pr)=\beta \int h \pr \ud x - \beta h(\bar{x})$ is a solution of
the DP equation~\eqref{eq:DP_static}. 
\end{romannum}

The sufficiency also follows similarly.  With any arbitrary
control $v_t$,
\begin{equation*}
\beta\int_{\Re^d} h\pr_0\ud x \leq \int_0^\infty {\sf L}(\pr_t,v_t)\ud t + \beta\limsup_{t \to \infty} \int_{\Re^d} h\pr_t\ud x , 
\end{equation*}
with equality if $v_t=u_t$ solves the pde~\eqref{eq:opt-cont-law}. 
Using the boundary condition, $\limsup_{t \to \infty} \int h\pr_t \ud
x = h(\bar{x})$, 
\begin{equation*}
J(u)=\beta \int h \pr_0 \ud x - \beta h(\bar{x}) \leq J(v).
\end{equation*} \qed

\newP{Proof of Lemma \ref{lem:L2norm}}  Suppose $\phi$ is the unique
solution of the Poisson equation~\eqref{eqn:EL_phi_intro}
(Theorem~\ref{thm:Poisson} in Appendix~\ref{apdx:Poisson}).  
Then $u=- \beta \nabla\phi$ is a particular solution of the
pde~\eqref{eq:opt-cont-law}.  The general solution is then given by
$u= v - \beta \nabla \phi$ where $v$ is a null solution, i.e.,
$\nabla \cdot(\pr v) =0$.  The $L^2$ optimality of the gradient
solution follows from the simple calculation:
\begin{align*}
\int |u|^2\pr \ud x &= \int \beta^2|\nabla\phi|^2\pr \ud x + \int |v|^2\pr \ud x -2\beta \int v \cdot \nabla \phi \pr \ud x\\
&= \beta^2\|\nabla \phi\|_2^2 + \|v\|^2_2,
\end{align*}
because $\int \nabla \phi \cdot v \pr \ud x = -\int \phi \nabla
\cdot(\pr v)\ud x = 0$. \qed

\subsection{Hamiltonian formulation}\label{apdx:Ham}

The Hamiltonian ${\sf H}$ is defined
in~\eqref{eq:H_Defn}. Suppose $u_t$ is the optimal control
and $\rho_t$ is the corresponding optimal trajectory.  Denote the
trajectory for the co-state (momentum) as ${\sf q}_t$.  Using the
Pontryagin's minimum principle, $(\rho_t,{\sf q}_t)$ satisfy the
following Hamilton's equations:
\begin{align*}
\frac{\partial \pr_t}{\partial t} &= \frac{\partial {\sf H}}{\partial
  \rho}(\rho_t,{\sf q}_t,u_t),\quad \rho_0 = p^*_0, \\
\frac{\partial {\sf q}_t}{\partial t} &= -\frac{\partial {\sf H}}{\partial
  {\sf q}}(\rho_t,{\sf q}_t,u_t),\quad {\sf q}_T = \frac{\partial}{\partial \rho} \left(\beta \int h(x)\rho(x)\ud x\right), \\
0 &= {\sf H}(\rho_t,{\sf q}_t,u_t) = \min_{v \in L^2} {\sf
  H}(\rho_t,{\sf q}_t,v). 
\end{align*} 

The calculus of variation argument in the proof of Theorem
\ref{thm:opt-cont} shows that the minimizing control $u_t$ solves the
first order optimality equation
\begin{equation}\label{eq:first-order-opt}
\frac{1}{\rho_t}\nabla \cdot (\rho_tu_t) = {\sf q}_t - \hat{\sf q}_{t},
\end{equation} 
where $\hat{\sf q}_t:=\int {\sf q}_t(x) \rho_t(x)\ud x$.  

The explicit
form of the Hamilton's equations are obtained by explicitly evaluating
the derivatives along the optimal trajectory:
\begin{align*}
\frac{\partial {\sf H}}{\partial \rho}(\rho_t,{\sf q}_t,u_t) &= - \nabla \cdot (\rho_t u_t),\\
\frac{\partial {\sf H}}{\partial {\sf q}}(\rho_t,{\sf q}_t,u_t)&=
\frac{\beta^2}{2}(h-\hat{h}_t)^2 - \frac{1}{2}({\sf q}_t-\hat{\sf q}_t)^2. 
\end{align*} 

It is easy to verify that 
${\sf q}_t \equiv  {\beta} h(x)$
satisfies both the boundary condition and the evolution equation for
the momentum.
This results in a simpler form of the Hamilton's equations:
\begin{align*}
\frac{\partial \pr_t}{\partial t}  &= - \nabla \cdot (\rho_t u_t),\\
0 &= {\sf H}(\rho_t, {\scriptstyle\beta} h,u_t) = \min_{v \in L^2} {\sf
  H}(\rho_t, {\scriptstyle\beta} h,v).
\end{align*}

In a particle filter implementation, the minimizing control $u_t=-\nabla\phi$ is obtained by
  solving the first order optimality
  equation~\eqref{eq:first-order-opt} with ${\sf q}_t=\beta h$. 

\subsection{Bayes' exactness and convergence}\label{apdx:exact}

Before proving the Theorem~\ref{thm:exact}, we state and prove the
following technical Lemma:

\medskip

\begin{lemma}\label{lem:spec_est}
Suppose the prior density $p_0^*(x)$ satisfies Assumption (A1) and the
objective function $h(x)$ satisfies assumption (A2).  Then for each
fixed time $t\geq 0$:
\begin{romannum}
\item The posterior density $p^*(x,t)$, defined according to
  \eqref{eqn:evolution}, admits a spectral bound;
\item The objective function $h\in L^2(\Re^d;p^*(\cdot,t))$.
\end{romannum}
\end{lemma}
\begin{proof}
Define $V_t(x) := -\log p^*(x,t) = V_0(x) + t\beta h(x) + \gamma_t$ where
$\gamma_t:=\log(\int e^{-V_0(y)-th(y)}\ud y)$.  It is directly verified that $V_t \in C^2$ with $D^2 V_t \in L^\infty$ and 
$\liminf_{|x| \to \infty}\;\;\nabla V_t(x) \cdot \frac{x}{|x|} =
\infty$.  Therefore, the density $p^*(x,t)$ admits a spectral
bound~[Thm 4.6.3 in~\cite{bakry2013}].  The function $h$ is
square-integrable because 
\[\int |h(x)|^2p^*(x,t)\ud x \leq e^{-\beta t\bar{h} - \gamma_t}\int |h(x)|^2e^{-V_0(x)}\ud x<\infty.\]
\qed
\end{proof}

\medskip

\newP{Proof of Theorem \ref{thm:exact}} 
Given any $C^1$ smooth and compactly supported (test) function $f$,
using the elementary chain rule,
\[
\ud f(X_t^i) = -\beta \nabla \phi(X_t^i) \cdot \nabla f(X_t^i).
\]
On integrating and taking expectations,
\begin{align*}
\expect{f(X_t^i)} & = \expect{f(X_0^i)} -\beta \int_0^t \expect{\nabla
  \phi(X_s^i) \cdot \nabla f(X_s^i)} \ud s\\
& = \expect{f(X_0^i)}  -\beta \int_0^t \expect{(h(X_s^i) - \hat{h}_s) f(X_s^i)} \ud s,
\end{align*}
which is the weak form of the replicator pde~\eqref{eq:replicator}.
Note that the weak form of the Poisson
equation~\eqref{eq:Poisson-weak}  is used to obtain the second
equality.  Since the test function $f$ is arbitrary, the evolution of
$p$ and $p^*$ are identical.  That the control function is
well-defined for each time follows from Theorem~\ref{thm:Poisson}
based on apriori estimates in Lemma~\ref{lem:spec_est} for $p^*=p$.  

\medskip

The convergence proof is presented next.  The proof here is somewhat
more general than needed to prove the Theorem.  For a function $h$, we
define the minimizing set:
\[
A_{0}:=\{x\in\Re^d \mid h(x) = \bar{h}\},
\]
where it is recalled that $\bar{h}=\inf_{x\in\Re^d} h(x)$. 
In the following it is shown that for {\em any} open neighborhood $U$
of $A_0$,
\begin{equation}
\liminf_{t \to \infty} \int_{U} p(x,t) \ud x = 1.
\label{eq:weakconv_of_p}
\end{equation}
It then follows that $X_t^i$ converges
in distribution 
where the limiting distribution is supported on $A_0$~[Thm. 3.2.5
in~\cite{durrett2010}].  If the minimizer is unique (i.e.,
$A_0=\{\bar{x}\}$), $X^i_t$ converges to $\bar{x}$ in probability.

The key to prove the convergence is the following property of the
function $h$: 

\newP{(P1)} For each $\delta > 0$, $\exists\;\epsilon > 0$ such that: 
\begin{equation*}
|h(x)-\bar{h}| \leq \epsilon\quad \Rightarrow \quad \text{dist}(x,A_0)\leq \delta\quad \forall x\in \Re^d,
\end{equation*}
where $\text{dist}(x,A_0)$ denotes the Euclidean distance of point
$x$ from set $A_0$.  If the
minimizer $\bar{x}$ is unique, it equals $|x-\bar{x}|$.
 
\medskip

Any lower semi-continuous function satisfying Assumption (A3) also
satisfies the property (P1): Suppose $\{x_n\}$ is a sequence such that
$h(x_n)\rightarrow \bar{h}$. Then $\{x_n\}$ is compact because
$h(x)>\bar{h}+r$ outside some compact set.  Therefore, the limit set
is non-empty and because $h$ is lower semi-continuous, for any limit
point $z$, $\bar{h} \leq h(z) \leq \liminf_{x_n\rightarrow z}
h(x_n)=\bar{h}$.  That is, $z\in A_0$.

The proof for~\eqref{eq:weakconv_of_p} is based on construction of a
Lyapunov function: Denote $A_\epsilon:=\{x\in\Re^d \mid h(x)
\leq \bar{h} + \epsilon\}$ where $\epsilon>0$.
By property (P1), given any open neighborhood $U$ containing $A_{0}$, $\exists \; \epsilon > 0$
such that $A_\epsilon \subset U$. 
A candidate Lyapunov function $V_{A_\epsilon}(\mu): =
-\beta^{-1}\log(\mu(A_\epsilon))$ is defined for measure $\mu$ with everywhere
positive density.  By construction $V_{A_\epsilon}(\mu) \geq 0$ with
equality iff $\mu(A_\epsilon)=1$. 

Let $\mu_t$ be the probability measure associated with $p(x,t)$, i.e,
$\mu_t(B):=\int_B p(x,t)\ud x$  for all Borel measurable set $B\subset
\Re^d$.  Since $p(x,t)$ is a solution of the replicator pde,
\begin{equation*}
\begin{aligned}
\frac{\ud}{\ud t} V_{A_\epsilon}(\mu_t)&= \frac{\ud}{\ud t} \big[ -\frac{1}{\beta}\log(\mu_t(A_\epsilon)) \big] \\ 
& = \frac{1}{\mu_t(A_\epsilon)} \int_{A_\epsilon} (h(x)-\hat{h}_t)\ud  \mu_t(x) 
\\
& = (1-\mu_t(A_\epsilon))\left( \frac{\int_{A_\epsilon} h\ud \mu_t}{\mu_t(A_\epsilon)}
  - \frac{\int_{A_\epsilon^c} h\ud \mu_t}{\mu_t(A_\epsilon^c)}\right)\\
  &\leq 0
\end{aligned}
\label{eq:lyp-ineq}
\end{equation*}
with equality iff $\mu_t(A_\epsilon)=1$. 

For the objective function $h$, a direct calculation also shows:
\begin{align*}
\frac{\ud}{\ud t} \int h(x)p(x,t)\ud x &= -\beta \int
h(x)(h(x)-\hat{h}_t)p(x,t)\ud x \\&= -\beta \int (h(x)-\hat{h})^2p(x,t)\ud x \leq 0,
\end{align*}
with equality iff $h = \hat{h}$ almost everywhere (with respect to
the measure $\mu_t$). \qed

\subsection{Quadratic Gaussian case}
\label{apdx:quadratic-Gaussian}

\newP{Proof of Lemma~\ref{lem:u-lin}} We are interested in obtaining
an explicit solution of the Poisson equation,
\begin{equation}
-\nabla\cdot(\rho(x) \nabla\phi(x)) = (h(x)-\hat{h})\rho(x).
\label{eq:PE_App_LG}
\end{equation}
Consider the solution ansatz:
\begin{equation}
\nabla\phi(x) = \K (x-m) + b,
\label{eq:QGansatz}
\end{equation}
where the matrix $\K=\K^T \in \Re^{d\times d}$ and the vector $b\in
\Re^d$ are determined as follows:
\begin{romannum}
\item Multiply both sides of~\eqref{eq:PE_App_LG} by vector $x$ and
  integrate (element-by-element) by parts to obtain
\begin{equation}
b =  \int x (h(x) - \hat{h}) \rho(x) \ud x.
\label{eq:b-integration-by-parts}
\end{equation}
\item Multiply both sides of~\eqref{eq:PE_App_LG} by matrix $(x-m)(x-m)^T$ and
  integrate by parts to obtain
\begin{equation}
\Sigma {\sf K} + {\sf K} \Sigma  = \int (x - m) (x-m)^T (h(x) -
   \hat{h}) \rho(x) \ud x.
\label{eq:K-integration-by-parts}
\end{equation}
\end{romannum}

We have thusfar not used the fact that the density $\rho$ is Gaussian
and the function $h$ is quadratic.  In the following, it is shown that
the solution thus defined in fact {\em solves} the pde~\eqref{eq:PE_App_LG}
under these conditions.  

A radially unbounded quadratic function is of the general form:
\begin{equation*}
h(x)=\frac{1}{2}(x-\bar{x})^TH(x-\bar{x}) + c
\end{equation*}
where the matrix $H = H^T \succ 0$ and $c$ is some constant.  For a
Gaussian density $\rho$ with mean $m$ and variance $\Sigma\succ 0$,
the integrals are explicitly evaluated to obtain
\begin{subequations}
\begin{align}
b & =  \int x (h(x) - \hat{h}) \rho(x) \ud x = \Sigma H(m-\bar{x}),
\label{eq:bexplict}\\
\Sigma {\sf K} + {\sf K} \Sigma  & = \int (x - m) (x-m)^T (h(x) -
   \hat{h}) \rho(x) \ud x = \Sigma H\Sigma. 
\label{eq:Kexplict}
\end{align}   
\end{subequations}
A unique positive-definite symmetric solution ${\sf K}$ exists for the
Lyapunov equation~\eqref{eq:Kexplict} because $\Sigma\succ 0$ and
$\Sigma H\Sigma\succ 0$~\cite{dullerud2013course}.  

On substituting the solution~\eqref{eq:QGansatz} into the
Poisson equation~\eqref{eq:PE_App_LG} and dividing through by $\rho$,
the two sides are:
\begin{align*}
& -\frac{1}{\rho}\nabla\cdot(\rho \nabla\phi)  = (x-m)^T\Sigma^{-1}(\K(x-m)+b) - \tr(\K),\\
& h - \hat{h}  = \frac{1}{2}(x-\bar{x})^TH(x-\bar{x}) -
\frac{1}{2}(m-\bar{x})^TH(m-\bar{x}) - \frac{1}{2}\tr(H\Sigma).
\end{align*}
where $\tr(\cdot)$ denotes the matrix trace. Using
formulae~\eqref{eq:bexplict}-\eqref{eq:Kexplict} for $b$ and $\K$, the
two sides are seen to be equal. 
\qed

\medskip

\newP{Proof of Proposition~\ref{prop:Gaussian}}
Using the affine control law~\eqref{eqn:particle_filter_lin}, the
particle filter is a linear system with a Gaussian prior:
\begin{equation}
\frac{\ud X^i_t}{\ud t} = -\beta \K_t(X^i_t-m_t) -\beta b_t,\quad
X^i_0\sim {\cal N}(m_0,\Sigma_0).
\label{eq:Xti_G}
\end{equation}
Therefore, the density of $X_t^i$ is Gaussian for all $t>0$.  The
evolution of the mean is obtained by taking an expectation of both sides
of the ode~\eqref{eq:Xti_G}:
\[
\frac{\ud}{\ud t}\expect{X^i_t}= -\beta b_t = -\beta \expect{X_t^i (h(X_t^i) - \hat{h}_t)},
\]
where~\eqref{eq:b-integral} is used to obtain the second equality.
The equation for the variance $\Sigma_t$ of $X_t^i$ is similarly obtained:
\begin{align*}
\frac{\ud \Sigma_t}{\ud t} & = -\beta (\K_t\Sigma_t + \Sigma_t\K_t)
\\
& = -\beta{\sf E} \left[ (X_t^i - m_t) (X_t^i-m_t)^T (h(X_t^i) -
   \hat{h}_t)\right], 
\end{align*}
where~\eqref{eq:K-integral} has been used.\qed

\medskip

\newP{Proof of Corollary~\ref{cor:Gaussian}} The closed-form
odes~\eqref{eq:Gaussian_KF_m} and~\eqref{eq:Gaussian_KF_s} are obtained by using
explicit formulae~\eqref{eq:bexplict} and~\eqref{eq:Kexplict} for $b$ and
$\K$, respectively.  
\qed
\subsection{Parametric case}
\label{apdx:parametric}
\newP{Proof of Theorem~\ref{prop:natural-grad}} The natural gradient
ode~\eqref{eq:natural-grad} is obtained by applying the chain
rule.  In its parameterized form, the density $p(x,t) =
\Varrho(x;\theta_t)$ evolves according to the replicator pde:
\[
\frac{\partial \Varrho}{\partial t} (x;\theta_t)  =  -\beta \,
(h(x)-\hat{h}_t) \,\Varrho(x;\theta_t).
\]
Now, using the chain rule,
\[
\frac{\partial \Varrho}{\partial t} (x,\theta_t) = \Varrho(x,\theta_t)
\left[ \frac{\partial
  }{\partial \vartheta} ( \log \Varrho (x;\theta_t) ) \right]^T \frac{\ud \theta_t}{\ud t},
\]
where $\frac{\partial }{\partial \vartheta} \left( \log \Varrho
\right)$ and $\frac{\ud \theta_t}{\ud t}$ are both $M\times 1$ column
vectors.  Therefore, the replicator pde is given by
\[
\left[ \frac{\partial
  }{\partial \vartheta} ( \log \Varrho (x;\theta_t) ) \right]^T
  \frac{\ud \theta_t}{\ud t} \Varrho(x;\theta_t) = - \beta (h(x)-\hat{h}_t) \Varrho(x;\theta_t) .
\]
Multiplying both sides by the column vector $\frac{\partial }{\partial
  \vartheta} \left( \log \Varrho \right)$, integrating over the
domain, and using the definitions~\eqref{eq:Fisher} of the Fisher information matrix $G$
and~\eqref{eq:grad-e} for $\nabla e$, one obtains
\[
G_{(\theta_t)} \frac{\ud \theta_t}{\ud t} = -\beta  \nabla e(\theta_t).
\]
The ode~\eqref{eq:natural-grad} is obtained because $G$ is assumed
invertible. 
\qed
\subsection{Galerkin approximation error}
\label{apdx:Galerkin}

\newP{Spectral representation} Under Assumptions (A1)-(A2), the spectrum of $-\Delta_\pr$
is known to be discrete with an ordered sequence of
eigenvalues $0=\lambda_0<\lambda_1\le\lambda_2\le\hdots$ and
associated eigenfunctions $\{e_n\}$ that form a complete orthonormal
basis of $L^2(\Re^d,\pr)$ [Corollary~4.10.9 in ~\cite{bakry2013}]. 
As a result, for $m,l\in \mathbb{Z}^+$:
\[
\big<e_k,e_l\big> = \delta_{kl},\quad
\big<\nabla e_k, \nabla e_l\big> =\lambda_m \delta_{kl}.
\]
The
trivial eigenvalue $\lambda_0=0$ with associated eigenfunction
$e_0(x)=1$.  On the subspace of zero-mean functions, the spectral representation yields: For $\phi\in L^2_0(\Re^d,\pr)$,
\begin{equation}
\label{eq:spectral_rep}
-\Delta_\rho \phi(x) = \sum_{k=1}^\infty \lambda_k \big<e_k,\phi\big>e_k(x).
\end{equation}

\newP{Proof of Proposition \ref{prop:galerkin}} By the triangle
inequality, 
\[
\|\nabla \phi-\nabla\phi^{(M,N)}\|_{2} \; \le\;
\underbrace{\|\nabla\phi-\nabla\phi^{(M)}\|_2}_{\text{bias}} \;+\; \|\nabla\phi^{(M)}-\nabla\phi^{(M,N)}\|_{2}. 
\]
The estimates for the bias and for the error due to the empirical
approximation are as follows:

\newP{Bias} 
Using the spectral representation~\eqref{eq:spectral_rep}, because $h-\hat{h}\in L^2_0$,
\[
\phi(x) = - \Delta_{\rho}^{-1} (h-\hat{h})(x) = \sum_{k=1}^{\infty} \frac{1}{\lambda_k}\big<e_k,h\big>e_k(x).
\]
With basis functions as eigenfunctions,
\[
\phi^{(M)}(x) = \sum_{k=1}^{M} \frac{1}{\lambda_k}\big<e_k,h\big>e_k(x).
\]
Therefore,
\begin{align*}
\|\nabla\phi - \nabla\phi^{(M)}\|^2_2 &= \sum_{k=M+1}^\infty \frac{1}{\lambda_k^2}
|\big<e_k,h\big>|^2 \|\nabla e_k\|_2^2\\
&= \sum_{k=M+1}^\infty\frac{1}{\lambda_k^2}
|\big<e_k,h\big>|^2\lambda_k \;\le\; \frac{1}{\lambda_M} \|h-\Pi_Sh\|^2_2,
\end{align*}
where $\Pi_S h(x):=\sum_{k=1}^{M} \big<e_k,h\big>e_k(x)$ denotes the
projection of $h$ onto $S$.

\newP{Empirical error} Suppose $\{X^i\}_{i=1}^N$ are drawn i.i.d. from
the density $\rho$.  The empirical solution is obtained as:
\begin{equation*}
\phiMN(x) = \sum_{k=1}^M \frac{1}{\lambda_k}\left(\frac{1}{N}\sum_{i=1}^N e_k(X^i)h(X^i)\right)e_k(x),
\end{equation*}
and the error,
\begin{equation*}
\phiMN(x)-\phiM(x) = \sum_{k=1}^M \frac{1}{\lambda_k}z_k^{(N)}e_k(x),
\end{equation*}
where $z_k^{(N)}:=\frac{1}{N}\sum_{i=1}^N e_k(X^i)h(X^i)- \big<e_k,h\big>$. 
Therefore,
\begin{equation}
\|\nabla\phiMN-\nabla\phiM\|_2^2 = \sum_{k=1}^M
\frac{1}{\lambda_k}|z_k^{(N)}|^2=:\epsilon_{N},
\label{eq:eps_app}
\end{equation}
where $\big<\nabla e_k,\nabla e_l\big>=\lambda_k \delta_{kl}$ is  used to
simplify the cross-terms.  Finally, by
applying the Law of Large Numbers (LLN) for the random variable $z_k^{(N)}$,
$\epsilon_{N}\stackrel{\text{a.s.}}{\rightarrow}0$ as $N\rightarrow \infty$. The LLN
applies because
\[\expect{|z_k^{(N)}|}\leq 2\,\expect{|e_k(X)h(X)|}\leq 2 \, \|h\|_2<\infty.\]

\newP{Variance} Under additional restrictions on $h$, one can obtain sharper estimates.
For example, taking the expectation of both sides of~\eqref{eq:eps_app},
\begin{align*}
\expect{\|\nabla\phiM-\nabla\phiMN\|_2^2} = \sum_{k=1}^M \frac{\expect{|z_k^{(N)}|^2}}{\lambda_k}.
\end{align*}
Now, $\expect{|z_k^{(N)}|^2} = \frac{\text{Var}(e_k(X^i)h(X^i))}{N}$.
Therefore, supposing $h\in L^\infty$, 
\begin{equation*}
\expect{\|\phiM-\phiMN\|_2^2} \leq \frac{\|h\|_\infty^2}{N}{\sum_{k=1}^M\frac{1}{\lambda_k}},
\end{equation*}
because $\text{Var}(e_k(X^i))=1$.  

\medskip

In summary, for bounded functions $h$,
\begin{equation*}
\expect{\|\nabla\phi - \nabla\phiMN\|_2}\; \leq \;\underbrace{\frac{
    \|h-\Pi_Sh\|_2}{\sqrt{\lambda_M}}}_{\text{bias}} \; +\;  
    \underbrace{\frac{\|h\|_\infty}{\sqrt{N}}\sqrt{\sum_{k=1}^M\frac{1}{\lambda_k}}}_{\text{variance}}. 
\end{equation*}\qed

\bibliographystyle{plain}
\bibliography{Optimization_abbr,ref,fpfbib}

\end{document}